\documentclass{amsart}

\usepackage{graphicx}
\usepackage{latexsym}
\usepackage{amssymb, latexsym}
\usepackage{amsmath}
\usepackage{mathrsfs}
\usepackage{amsxtra}
\usepackage{amsthm}
\usepackage{url}
\usepackage{verbatim}
\newtheorem{theorem}{Theorem}[section]
\newtheorem{corollary}[theorem]{Corollary}

\newtheorem{lemma}[theorem]{Lemma}

\newtheorem{proposition}[theorem]{Proposition}

\theoremstyle{definition}
\newtheorem{definition}[theorem]{Definition}

\newtheorem{question}[theorem]{Question}

\newcommand{\image}{\mathbin{\hbox{\tt\char'42}}}

\newcommand{\smallleq}{\mathrel{\mathchoice{\raise2pt\hbox{$\scriptstyle\leq$}}{\raise1pt\hbox{$\scriptstyle\leq$}}{\raise1pt\hbox{$\scriptscriptstyle\leq$}}{\scriptscriptstyle\leq}}}
\newcommand{\smalllt}{\mathrel{\mathchoice{\raise2pt\hbox{$\scriptstyle<$}}{\raise1pt\hbox{$\scriptstyle<$}}{\raise0pt\hbox{$\scriptscriptstyle<$}}{\scriptscriptstyle<}}}
\newcommand{\Add}{\mathop{\rm Add}}

\newcommand{\GCH}{{\rm GCH}}

\newcommand{\ZFC}{{\rm ZFC}}

\newcommand{\la}{\langle}
\newcommand{\ra}{\rangle}

\newcommand{\crit}{\text{crit}}

\newcommand{\dom}{\text{dom}}

\newcommand{\Ord}{{\rm Ord}}
\newcommand{\cL}{{\mathbb L}}

\newcommand{\ULS}{{\mathrm{ULST}}}
\newcommand{\SULS}{{\mathrm {SULST}}}
\newcommand{\sL}{{\mathcal L}}

\newcommand{\Tr}{\rm{Tr}}

\author{Victoria Gitman}
\address[V. Gitman]{Mathematics Department, CUNY Graduate Center, NY, USA}
\email{vgitman@gmail.com}
\author{Jonathan Osinski}
\address[J. Osinski]{Department of Mathematics, University of Hamburg, Bundesstraße 55, 20146 Hamburg, Germany}
\email{jonathan.osinski@uni-hamburg.de}

\title{Upward L\"owenheim-Skolem-Tarski numbers for abstract logics}
\begin{document}
\begin{abstract}
Galeotti, Khomskii and V\"a\"an\"anen recently introduced the notion of the upward L\"owenheim-Skolem-Tarski number for a logic, strengthening the classical notion of a Hanf number. A cardinal $\kappa$ is the \emph{upward L\"owenheim-Skolem-Tarski number} ($\ULS$) of a logic $\mathcal L$ if it is the least cardinal with the property that whenever $M$ is a model of size at least $\kappa$ satisfying a sentence $\varphi$ in $\mathcal L$, then there are arbitrarily large models satisfying $\varphi$ and having $M$ as a substructure. The substructure requirement is what differentiates the $\ULS$ number from the Hanf number and gives the notion large cardinal strength. While it is a theorem of $\ZFC$ that every logic has a Hanf number, Galeotti, Khomskii and V\"a\"an\"anen showed that the existence of the $\ULS$ number for second-order logic implies the existence of a partially extendible cardinal. We answer positively their conjecture that the $\ULS$ number for second-order logic is the least extendible cardinal.

We define the \emph{strong $\ULS$} number by strengthening the substructure requirement to elementary substructure. We investigate the $\ULS$ and strong $\ULS$ numbers for several classical strong logics: infinitary logics, the equicardinality logic, logic with the well-foundedness quantifier, second-order logic, and sort logics. We show that the $\ULS$ and the strong $\ULS$ numbers are characterized in some cases by classical large cardinals and in some cases by natural new large cardinal notions that they give rise to. We show that for some logics the notions of the $\ULS$ number, strong $\ULS$ number and least strong compactness cardinal coincide, while for others, it is consistent that they can be separated. Finally, we introduce a natural large cardinal notion characterizing strong compactness cardinals for the equicardinality logic.
\end{abstract}
\maketitle
\section{Introduction}
First-order logic relies, however minimally, on the set-theoretic background universe in which we work because it uses properties of the natural numbers. Tarski's definition of truth, for instance, requires recursion. The dependence on the set-theoretic background universe becomes more apparent when we consider stronger logics, such as infinitary logics and second-order logic. In these logics, we can express more properties of the models at the cost of having to use more of the set-theoretic universe to define the logic itself. Niceness properties of strong logics have been shown to be connected to and often equivalent to the existence of large cardinals. For many classical strong logics, for instance, the existence of a strong compactness cardinal is equivalent to the existence of some classical large cardinal. In this article, we investigate upward L\"owenheim-Skolem principles for various classical logics\footnote{We will give a formal definition of a logic in Section~\ref{sec:abstractLogics}.} and show that these principles are equivalent to the existence of large cardinals, some classical and some new.

The \emph{upward L\"owenheim-Skolem Theorem} for first-order logic says that any infinite structure has arbitrarily large elementary superstructures. A very weak version of the upward L\"owenheim-Skolem Theorem holds for all logics. Recall that the \emph{Hanf number} of a logic $\sL$ is the least cardinal $\delta$ such that for every language $\tau$ and $\mathcal L(\tau)$-sentence $\varphi$, if a $\tau$-structure $M\models_{\sL}\varphi$ has size $\gamma\geq\delta$, then for every cardinal $\overline\gamma>\gamma$, there is a $\tau$-structure $\overline M$ of size at least $\overline\gamma$ such that $\overline M\models_{\sL}\varphi$. The upward L\"owenheim-Skolem Theorem implies, in particular, that the Hanf number of first-order logic is $\omega$. $\ZFC$ proves that every logic has a Hanf number (see, for instance, \cite[Chapter II, Theorem 6.1.4]{Barwise:ModelTheoreticLogics}). Galeotti, Khomskii and V\"a\"an\"anen strengthened the notion of the Hanf number by requiring that the model $\overline M\models_{\sL}\varphi$ of size at least $\overline\gamma$ is a (not necessarily elementary) superstructure of $M$.

\begin{definition}[\cite{GaleottiKhomskiiVaananen:upwardLowenheimSkolem}]\label{def:ULS}
Fix a logic $\sL$. The \emph{upward L\"owenheim-Skolem-Tarski number} $\ULS(\sL)$, if it exists, is the least cardinal $\delta$ such that for every language $\tau$ and $\mathcal L(\tau)$-sentence $\varphi$, if a $\tau$-structure $M\models_{\sL}\varphi$ has size $\gamma\geq\delta$, then for every cardinal $\overline\gamma>\gamma$, there is a $\tau$-structure $\overline M$ of size at least $\overline\gamma$ such that $\overline M\models_{\sL}\varphi$ and $M\subseteq \overline M$ is a substructure of $\overline M$.
\end{definition}
\noindent The differences in the definitions of the Hanf and ULST numbers mirror differences in notions of varying strengths generalizing the \emph{downward} L\"owenheim-Skolem Theorem. The \emph{L\"owenheim-Skolem} number $\text{LS}(\mathcal L)$ of a logic $\mathcal L$ is the smallest cardinal $\kappa$ such that any satisfiable sentence $\varphi$ of $\mathcal L$ has a model of size smaller than $\kappa$. Similar to the ULST number, the L\"owenheim-Skolem-\emph{Tarski} number $\text{LST}(\mathcal L)$ adds the requirement, that any model of $\varphi$ has to have a substructure of size smaller than $\kappa$ satisfying $\varphi$. As with the Hanf number, it can be shown in $\ZFC$ that every logic has an LS number. On the other hand, there is a long history going back to \cite{Magidor:Compactness} that shows that $\text{LST}$ numbers of logics can have large cardinal strength (see \cite{MagidorVaananen:LSTnumbers}).

Similar to this gain in strength that occurs when switching attention from LS numbers to LST numbers, Galeotti, Khomskii and V\"a\"an\"anen showed that the existence of the $\ULS$ number for second-order logic, $\ULS(\cL^2)$, implies the existence of very strong large cardinals.
\begin{theorem}[\cite{GaleottiKhomskiiVaananen:upwardLowenheimSkolem}]
If $\ULS(\cL^2)$ exists, then for every $n \in  \omega$ there is an $n$-extendible cardinal $\lambda \leq \ULS(\cL^2)$.
\end{theorem}
\noindent They further conjectured that the strength of the existence of $\ULS(\cL^2)$ is exactly that of an extendible cardinal. We positively answer their conjecture.
\begin{theorem}
If $\ULS(\cL^2)$ exists, then it is the least extendible cardinal.
\end{theorem}

We can further strengthen the notion of the $\ULS$ number to capture the full power of the upward L\"owenheim-Skolem Theorem.
\begin{definition}
Fix a logic $\sL$. The \emph{strong upward L\"owenheim-Skolem-Tarski number} $\SULS(\sL)$, if it exists, is the least cardinal $\delta$ such that for every language $\tau$ and every $\tau$-structure $M$ of size $\gamma\geq\delta$, for every cardinal $\overline\gamma > \gamma$, there is a $\tau$-structure $\overline M$ of size at least $\overline\gamma$ such that $M\prec_{\sL} \overline M$ is an $\sL$-elementary substructure of $\overline M$.
\end{definition}

\noindent Notice that we could equivalently define the strong upward L\"owenheim-Skolem-Tarski number analogously to the upward L\"owenheim-Skolem-Tarski number but preserving theories instead of single sentences.

We pinpoint large cardinal notions that are equivalent to the existence of the $\ULS$ and strong $\ULS$ numbers for second-order logic, V\"a\"an\"anen's $\Sigma_n$-sort logics $\cL^{s,n}$, the logic $\cL(Q^{WF})$ - first-order logic augmented with the well-foundedness quantifier, the infinitary logics $\cL_{\kappa,\kappa}$, and the logic $\cL(I)$ - first-order logic augmented with the equicardinality quantifier.

\section{Abstract logics}\label{sec:abstractLogics}
The concept of an abstract logic abstracts away the properties of what it means to reasonably assign truth values to statements of a formal system interpreted in structures over a given language. Informally, a logic consists of an assignment of formulas to every language and a satisfaction relation telling us which of the formulas a given structure of the language satisfies. In first-order logic, for instance, the formulas consist of the atomic formulas closed under negation, conjunction, disjunction, and quantification and the satisfaction relation is given by Tarski's definition of truth. The notion of a formula of an abstract logic and the satisfaction relation must be constrained by properties that we expect a reasonable notion of formulas and satisfaction of them to obey.

A \emph{logic} is a pair of classes $(\sL, \models_\sL)$. The class  $\sL$ is a function on the class of all languages, and we call the image $\sL(\tau)$ of a language $\tau$ the set of $\tau$-sentences.  The class $\models_\sL$ is the satisfaction relation consisting of pairs $(M,\varphi)$, where \hbox{$\varphi\in \sL(\tau)$} and $M$ is a $\tau$-structure. The relation $\models_\sL$ is required to respect reducts, isomorphisms, and renamings of a language. If $\sigma\subseteq \tau$ are languages, then every $\sL(\sigma)$-sentence is an $\sL(\tau)$-sentence and if a $\tau$-structure $M$ satisfies an $\sL(\sigma)$-sentence $\varphi$ (according to $\models_{\sL}$), then it also satisfies the sentence $\varphi$ as the reduct $\sigma$-structure. If $\tau$-structures $M$ and $N$ are isomorphic, then they satisfy the same $\sL(\tau)$-sentences. Next, we define that $f$ is a \emph{renaming} between languages $\sigma$ and $\tau$ if it is an arity-preserving bijective map between the functions and the relations of the languages and a bijective map between the constants. Informally, a renaming renames the relations, functions, and constants of the language $\sigma$ by different names resulting in the language $\tau$. A renaming between two languages $\sigma$ and $\tau$ induces an obvious way to turn every $\sigma$-structure into a $\tau$-structure, and vice versa. We demand that a renaming induces a bijection between $\mathcal L(\tau)$ and $\sL(\sigma)$ and if the bijection associates an $\sL(\sigma)$-sentence $\varphi$ with an $\sL(\tau)$-sentence $\overline\varphi$, then a $\sigma$-structure $M$ should satisfy $\varphi$ if and only if it satisfies $\overline\varphi$ as the associated $\tau$-structure. Finally, for technical reasons, we require that our logics have an occurrence number. The occurrence number property captures our intuition that there should be a bound on the number of elements of a language that a single assertion can reference. The \emph{occurrence number} of a logic $\sL$ is the least cardinal $\kappa$ such that for every $\varphi \in \sL(\tau)$, there is $\overline\tau\subseteq\tau$ of size less than $\kappa$ such that $\varphi \in \sL(\overline\tau)$. We only allow set-sized languages and the collection of all sentences for a given language is required to be a set. For a fully formal definition of an abstract logic, see, e.g., \cite{BagariaDimopoulosGitmanMagidor:LargeCardinalLogics}.

The occurrence number requirement plays an important role in proofs involving properties of abstract logics. Without it, we can end up with logics such as $\cL_{\Ord,\Ord}$, which, for instance, can never have a strong compactness cardinal. Although the definition of an abstract logic only mentions sentences, it is easy to extend it to formulas by introducing and interpreting constants. Using this observation, we will assume that our logics can handle formulas with free variables. Additional properties that can be included in the definition of an abstract logic, which we don't assume here, include, for example, closure under Boolean connectives and quantifiers.

In this article, we will work with several classical abstract logics, whose definitions and properties we will now review.  To distinguish an abstract logic from a specific logic, we use $\sL$ to denote abstract logics and $\cL$ (with some decoration) to denote specific logics.

Given regular cardinals $\mu\leq\kappa$, the infinitary logic $\cL_{\kappa,\mu}$ extends first-order logic by closing the rules of formula formation under conjunctions and disjunctions of ${<}\kappa$-many formulas that are jointly in ${<}\mu$-many free variables, and under quantification of ${<}\mu$-many variables. $\cL_{\omega,\omega}$, usually written simply as $\cL$, is first-order logic. While first-order logic relies on the properties of natural numbers, the infinitary logics expand on this by relying on the properties of the ordinals. In the logic $L_{\omega_1,\omega}$, in an arithmetic or set theoretic structure, we can express that the natural numbers are standard. In the logic $\cL_{\omega_1,\omega_1}$, we can express that a binary relation is well-founded. In the logic $\cL_{\kappa,\omega}$, for every ordinal $\xi<\kappa$ and definable binary relation $\psi(y,x)$, there is a corresponding formula $\varphi^\xi_{\psi}(x)$ expressing that the relation given by $\psi$ when restricted to the predecessors of $x$ is isomorphic to the well-order $(\xi,\in)$. Thus, in particular, every ordinal $\xi<\kappa$ is definable in a transitive model of set theory in the logic $\cL_{\kappa,\omega}$. For every $\alpha<\kappa$, there is a sentence $\psi_\alpha$ in the logic $\cL_{\kappa,\kappa}$, which over a transitive model of set theory $N$, expresses closure under $\alpha$-sequences, $N^{\alpha}\subseteq N$.

Second-order logic $\cL^2$ extends first-order logic by allowing quantification over all relations on the underlying set of the model from the background set-theoretic universe. In structures where coding is available, such as arithmetic or set theory, this reduces to quantification over all subsets of the underlying set. In the logic $\cL^2$, we can express that a definable binary relation $\psi(y,x)$ is well-founded and moreover for every $x$, we can express that the relation given by $\psi$ on the predecessors of $x$ is isomorphic to a rank initial segment $(V_\alpha,\in)$ for some ordinal $\alpha$. Thus, in particular, given a set in a transitive model of set theory, we can express that the set is some rank initial segment $V_\alpha$.

The logic $\cL(Q^{WF})$ is first-order logic augmented by the quantifier $Q^{WF}$ that takes in two variables $x$ and $y$ so that $Q^{WF}xy\varphi(x,y)$ is true whenever $\varphi(x,y)$ defines a well-founded relation. Note that $\cL(Q^{WF})\subseteq \cL_{\omega_1,\omega_1}$ and also $\cL(Q^{WF})\subseteq  \cL^2$.

The logic $\cL(I)$ is first-order logic augmented by the \emph{H\"artig} or \emph{equicardinality} quantifier $I$ that takes in two variables $x$ and $y$ and two formulas $\varphi(x)$ and $\psi(y)$ so that $Ixy\varphi(x)\psi(y)$ is true whenever the sets defined by $\varphi(x)$ and $\psi(y)$ have the same cardinality in the background set-theoretic universe. Note that \hbox{$\cL(I)\subseteq \cL^2$}. In the logic $\cL(I)$, in an arithmetic or set-theoretic structure, we can express that the natural numbers are standard (unlike a standard natural number, the predecessors of a nonstandard natural number are bijective with the predecessors of its successor). More generally, we can express that a set-theoretic structure is \emph{cardinal correct} in the sense that for any predecessor of a cardinal, the set of all its predecessors is smaller in cardinality than the set of the predecessors of the cardinal.

The sort logics $\cL^{s,n}$, for $n<\omega$, introduced by V\"a\"an\"anen, are some of the most strong logics available (see \cite{vaananen:sortLogic} for a precise definition and properties). The logic $\cL^{s,n}$ is an extension of second-order logic and deals with many-sorted structures, i.e., structures that are possibly equipped with multiple sorts, each with its own universe of objects and possible relations between them. The key features of $\cL^{s,n}$ are the \emph{sort quantifiers} $\tilde\exists$ and $\tilde\forall$ which  range over all sets in $V$ searching for additional sorts with which we can expand the model to satisfy some desired property for how the old and new sorts should look and interact. We are however restricted to $n$-alternations of the sort quantifiers because if we allowed arbitrary formulas with these quantifiers, the existence of the satisfaction relation would violate Tarski's undefinability of truth. Because $\cL^{s,n}$ extends second-order logic, we can pick out the structures (isomorphic to) $(V_\alpha,\in)$, and now, using the sort quantifiers, it is possible to express that $V_\alpha$ is $\Sigma_n$-elementary in $V$ (see, for example, \cite{BagariaDimopoulosGitmanMagidor:LargeCardinalLogics} on how to do this).
\section{Generalized compactness}
Given a logic $\sL$, an $\sL$-theory (a set of $\sL$-sentences in a fixed language) is said to be \emph{${<}\kappa$-satisfiable} when every ${<}\kappa$-sized subset of it has a model. A cardinal $\kappa$ is a \emph{strong compactness cardinal} of a logic $\sL$ if every ${<}\kappa$-satisfiable $\sL$-theory is satisfiable. Observe that if a cardinal $\kappa$ is a strong compactness cardinal for a logic $\sL$, then indeed every cardinal $\gamma>\kappa$ is also a strong compactness cardinal for $\sL$.

The \emph{compactness Theorem} states that $\omega$ is a strong compactness cardinal of first-order logic. In a pioneering result in this area, Tarski showed that a cardinal $\kappa$ is a strong compactness cardinal of $\cL_{\kappa,\kappa}$ if and only if it is strongly compact \cite{Tarski:Compactness}. Magidor showed that the least extendible cardinal is the least strong compactness cardinal for second-order logic $\cL^2$ \cite{Magidor:Compactness}. Magidor also showed that a cardinal $\kappa$ is a strong compactness cardinal for the logics $\cL_{\omega_1,\omega_1}$ or $\cL(Q^{WF})$ if and only if it is $\omega_1$-strongly compact (see Section~\ref{sec:WF} for details). Surprisingly  little is known about strong compactness cardinals for the logic $\cL(I)$.

Most generally, Makowsky showed that every logic has a strong compactness cardinal if and only if Vop\v{e}nka's Principle holds \cite[Theorem 2]{Makowsky:Vopenka}.

It is not difficult to see, for a logic $\sL$, that the least strong compactness cardinal, if it exists, implies the existence of and bounds to the strong $\ULS$ number.
\begin{proposition}\label{prop:compactImpliesSULS}
If a logic $\sL$ has a strong compactness cardinal $\kappa$, then $\SULS(\sL)$ exists and is at most $\kappa$.
\end{proposition}
\begin{proof}
Suppose that $\kappa$ is a strong compactness cardinal of $\sL$. Fix a $\tau$-structure $M$ of size $\gamma\geq\kappa$ and a cardinal $\overline\gamma>\gamma$. Let $\tau'$ be the language $\tau$ extended by adding $\overline\gamma$-many constants $\{c_\xi\mid\xi<\overline\gamma\}$. Let $T$ be the $\sL(\tau')$-theory consisting of the $\sL$-elementary diagram of $M$ and assertions $c_\xi\neq c_\eta$ for $\xi<\eta<\overline\gamma$. Clearly, the theory $T$ is ${<}\kappa$-satisfiable because it holds true in $M$ (with the distinct constants $c_\xi$ mapped to distinct elements of $M$). Thus, $T$ has a model $N$. By construction, $N$ has size at least $\overline\gamma$ and is an $\sL$-elementary superstructure of $M$.
\end{proof}
It follows, by Makowsky's theorem, that if Vop\v enka's Principle holds, then every logic $\sL$ has a strong $\ULS$-number.

\section{Truth predicates}
Suppose that $(M,\in)$ is a transitive model closed under the pairing function. We shall say that $\Tr^M\subseteq M$ is a \emph{truth predicate} for $M$ if for all first-order formulas $\varphi(x_1, \dots, x_n)$ and all tuples $(a_1, \dots, a_n)$: $$(M, \in) \models \varphi(a_1, \dots, a_n) \text{ if and only if } (\varphi, a_1, \dots, a_n) \in \Tr^M.$$ There is a sentence $\varphi_{\text{truth}}$ of first-order logic such that for any $T\subseteq M$, $$(M,\in, T)\models\varphi_{\text{truth}}\text{ if and only if }T=\Tr^M.$$ The sentence $\varphi_\text{truth}$ can be obtained by taking the conjunction of the following sentences, which go through Tarski's truth definition in the usual way:
	\begin{enumerate}
		\item $\forall x ( x= (\ulcorner x_i\in x_j \urcorner, a_1, \dots, a_n) \rightarrow (T(x) \leftrightarrow a_i\in a_j))$.
		\item $\forall x ( x= (\ulcorner x_i = x_j\urcorner, a_1, \dots, a_n) \rightarrow (T(x) \leftrightarrow a_i = a_j))$.
		\item $\forall x (x =(\ulcorner \psi \wedge \chi \urcorner, a_1, \dots, a_n) \rightarrow (T(x) \leftrightarrow T((\ulcorner \psi \urcorner, a_1, \dots, a_n)) \wedge T((\ulcorner \chi \urcorner, a_1, \dots, a_n))))$.
		\item $\forall x (x =(\ulcorner \neg \psi \urcorner, a_1, \dots, a_n) \rightarrow (T(x) \leftrightarrow \neg T((\ulcorner \psi \urcorner, a_1, \dots, a_n))))$.
		\item $\forall x (x =(\ulcorner \exists x \psi \urcorner, a_1, \dots, a_n) \rightarrow (T(x) \leftrightarrow \exists y T((\ulcorner \psi \urcorner, a_1, \dots, a_n, y))))$.
	\end{enumerate}

Suppose that $(M,\in)$ and $(N,\in)$ are transitive models, $(M,\in, \Tr^M)\models\varphi_{\text{truth}}$ and $(N,\in,\Tr^N)\models\varphi_{\text{truth}}$. If there is an embedding

 $$j: (M, \in, \Tr^M) \rightarrow (N, \in, \Tr^N),$$  then $j$ is an elementary embedding between the structures $(M,\in)$ and $(N,\in)$. Thus, by including a truth predicate in our set-theoretic structures, we essentially get elementarity for free from just being a substructure. We will repeatedly use this simple observation to extract large cardinal strength from the existence of $\ULS$ numbers.

\section{The logic $\cL(Q^{WF})$}\label{sec:WF}
In this section we show that $\ULS(\cL(Q^{WF}))=\SULS(\cL(Q^{WF}))$ is the least measurable cardinal.

\begin{theorem}\label{th:measImpliesSULSWF}
If there is a measurable cardinal $\kappa$, then $\SULS(\cL(Q^{WF}))$ exists and is at most $\kappa$.
\end{theorem}
\begin{proof}
Suppose that $\kappa$ is a measurable cardinal. Suppose that $N$ is a $\tau$-structure of size $\gamma\geq\kappa$. Let $\overline\gamma>\gamma$ be any cardinal. Let $j:V\to M$ be an elementary embedding with $\crit(j)=\kappa$ and $j(\kappa)>\overline\gamma^+$. We can obtain such an embedding by iterating the ultrapower construction with a $\kappa$-complete ultrafilter enough times. Consider the $j(\tau)$-structure $j(N)\in M$. Since $j\image\tau\subseteq\tau$, (a reduct of) $j(N)$ is a $\tau$-structure modulo the renaming which takes $\tau$ to $j\image\tau$. The renaming gives rise to the associated bijection between formulas over $\tau$ and $j\image\tau$, so specifically, let it send $\varphi$ to $\overline\varphi$. It is easy to see that $\overline N=j\image N\subseteq j(N)$ is a $\tau$-substructure of $j(N)$. Clearly, the bijection $j$ witnesses that $N\cong \overline N$ as $\tau$-structures because, for instance, if $R(x)$ is a relation in $\tau$ and \hbox{$N\models R(a)$}, then, by elementarity, $j(R)(j(a))$ holds in $j(N)$, so $j(N)\models R(j(a))$ via the renaming, and so $\overline N\models R(j(a))$. Next, let's show that $\overline N\prec_{\cL(Q^{WF})} j(N)$. Suppose that $$\overline N\models_{\cL(Q^{WF})}\varphi(j(a)).$$ Via the isomorphism, it follows that $$N\models_{\cL(Q^{WF})}\varphi(a).$$ But then by elementarity of $j$, $M$ satisfies that $j(N)\models_{\cL(Q^{WF})}\overline\varphi(j(a))$ (as a $j(\tau)$-structure). Since $M$ is well-founded, and thus correct about the well-foundedness quantifier, and first-order satisfaction is absolute, it is actually the case that $$j(N)\models_{\cL(Q^{WF})}\overline\varphi(j(a))$$ as a $j\image\tau$-structure, and hence $$j(N)\models_{\cL(Q^{WF})}\varphi(j(a))$$ modulo the renaming as a $\tau$-structure. Finally, since $|N|=\gamma\geq\kappa$, $M$ satisfies that $|j(N)|\geq j(\kappa)>\overline\gamma^+$, but then in $V$, $|j(N)|\geq \overline\gamma^+>\overline\gamma$ as desired.
\end{proof}

\begin{theorem}\label{th:ULSWF}
If $\ULS(\cL(Q^{WF}))=\delta$, then there is a measurable cardinal ${\leq}\delta$.
\end{theorem}
\begin{proof}
Suppose $\ULS(\cL(Q^{WF}))=\delta$. Consider the model $$\mathcal M=( H_{\delta^+},\in,\delta,\Tr),$$ where $\Tr$ is a truth predicate for $(H_{\delta^+},\in)$. Then $\mathcal M$ satisfies the sentence $\varphi$ in the logic $\cL(Q^{WF})$, which is the conjunction of the sentences:
\begin{enumerate}
\item I am well-founded.
\item $\delta$ is the largest cardinal.
\item $\Tr$ is a truth predicate for $(H_{\delta^+},\in)$.
\end{enumerate}
Since $\ULS(\cL(Q^{WF}))=\delta$, there is a structure $$\mathcal N=(N,{\mathsf E},\overline\delta,\overline \Tr)$$ of size much larger than $\delta$ having $\mathcal M$ as a substructure and which is a model of the above sentences. Since $\mathcal N$ is well-founded, we can assume, by collapsing, that ${\mathsf E}=\in$, $N$ is transitive, and there is an embedding $$j:H_{\delta^+}\to N$$ such that $j(\delta)=\overline\delta$. Observe that since $\bar\delta$ is the largest cardinal of $\mathcal N$ and $|N|$ is much larger than $\delta$, it follows that $\overline\delta>\delta$. Since we have included a truth predicate, the embedding $j$ is elementary, and since $j(\delta)=\overline\delta>\delta$, $j$ has a critical point $\crit(j)=\kappa\leq\delta$. Since the powerset of $\kappa$ is contained in $H_{\delta^+}$, we can use $j$ to derive a $\kappa$-complete ultrafilter on $\kappa$ witnessing that $\kappa$ is measurable.
\end{proof}
\begin{corollary}\label{cor:ULSWF} The following are equivalent for a cardinal $\kappa$.
\begin{enumerate}
\item $\kappa$ is the least measurable cardinal.
\item $\kappa=\ULS(\cL(Q^{WF}))$.
\item $\kappa=\SULS(\cL(Q^{WF}))$.
\end{enumerate}
\end{corollary}
Next, we would like to understand the relationship between the least strong compactness cardinal for $\cL(Q^{WF})$ and $\SULS(\cL(Q^{WF}))$. Recall that a cardinal $\delta$ is \emph{$\gamma$-strongly compact} for $\gamma \leq \delta$ if every $\delta$-complete filter (on any set) can be extended to a $\gamma$-complete ultrafilter. Note that this is not a typical large cardinal notion because if $\delta$ is $\gamma$-strongly compact, then it is easy to see that every cardinal $\delta'\geq\delta$ is also $\gamma$-strongly compact. This, property is, however, necessary for the theorem below because the same holds for strong compactness cardinals.

\begin{theorem}[Magidor]\label{th:omega1StronglyCompactCompactnessCardinal}
A cardinal $\delta$ is a strong compactness cardinal for $\cL(Q^{WF})$ if and only if $\delta$ is $\omega_1$-strongly compact.
\end{theorem}
We are not aware of any published proof of this result, so we give one here.
\begin{proof}
Suppose that $\delta$ is $\omega_1$-strongly compact. Fix a language $\tau$. Let $T$ be a ${<}\delta$-satisfiable theory in $\cL(Q^{WF})(\tau)$. We have to find a model of $T$. By\break \cite[Theorem 1.2]{BagariaMagidor:Omega1StronglyCompact}, $\omega_1$-strong compactness of $\delta$ gives us an elementary embedding $j: V \to M$ with \hbox{$\crit(j)=\kappa \geq \omega_1$} and $d \in M$ such that $j \image T \subseteq d$ and $M \models |d|<j(\delta)$. In $M$, let $S=j(T)\cap d$. Observe that $M \models |S| < j(\delta)$ and $j \image T \subseteq S \subseteq j(T)$. By elementarity, $M$ believes that every subset of $j(T)$, a theory in $\cL(Q^{WF})(j(\tau))$ of size ${<} j(\delta)$, is satisfiable. It follows that $M$ has a $j(\tau)$-structure $\mathcal A \models S$. As $M$ is transitive, it is correct about the well-foundedness quantifier, so $\mathcal A$ is really a model of $S$ and, in particular, of $j \image T \subseteq S$. Using the renaming  which takes $\tau$ to $j\image\tau$, we get that $\mathcal A$ is a model of $T$ as a $\tau$-structure.

Next, suppose that $\delta$ is a strong compactness cardinal for $\cL(Q^{WF})$. To show that $\delta$ is $\omega_1$-strongly compact, by the proof of \cite[Theorem 4.7]{BagariaMagidor:groupRadicals} it is sufficient to produce for every $\alpha \geq \delta$ a fine $\omega_1$-complete ultrafilter on $\mathcal P_\delta(\alpha)$. If $\gamma$ is an ordinal with $\mathcal P(\alpha) \in V_\gamma$ and we have an elementary embedding $j: V_\gamma \to M$ with $M$ transitive, $\crit(j) \geq \omega_1$, $d\in M$ with $j \image \alpha \subseteq d \subseteq j(\alpha)$ and $M \models |d| < j(\delta)$, then it is routine to check that $U$ defined by
	\[
	X \in U \text{ iff } X \subseteq \mathcal P_\delta(\alpha) \text{ and } d \in j(X)
	\]
	is a fine $\omega_1$-complete ultrafilter on $\mathcal P_\delta(\alpha)$. Let $\tau$ be the language consisting of a binary relation $\in$ and constants $\{c_x\mid x\in V_\gamma\}\cup \{d\}$. Let $T$ be the following $\cL(Q^{WF})(\tau)$-theory:
$${\rm ED}_{\cL(Q^{WF})}(V_\gamma,\in,c_x)_{x\in V_\gamma}\cup\{c_\xi\in d\mid \xi<\alpha\}\cup\{|d|<c_\delta\},$$
where ${\rm ED}_{\cL(Q^{WF})}$ stands for the elementary diagram and each $c_x$ is interpreted by the associated set $x$.

 The theory $T$ is clearly ${<}\delta$-satisfiable as witnessed by the model $(V_\gamma,\in,c_x)_{x\in V_\gamma}$. So $T$ has a well-founded model $M$, and by collapsing we can assume without loss of generality that $M$ is transitive. Thus, we get an elementary embedding $j:V_\gamma\to M$ with $\crit(j)\geq\omega_1$ because the embedding is into a well-founded target. Letting $d^M$ be the interpretation of $d$ in $M$, let $d'=d^M\cap j(\alpha)$. Then, in $M$, $|d'|\leq |d^M|<j(\delta)$ and $j\image \alpha\subseteq d'$ as desired.

\end{proof}
It is not difficult to see that if $\delta$ is $\omega_1$-strongly compact, then there is a measurable cardinal ${\leq}\delta$. Magidor showed in \cite{Magidor:LeastMeasurableStronglyCompact} that it is consistent, relative to a supercompact cardinal, that the least measurable cardinal is the least strongly compact cardinal, and hence, in particular, the least measurable cardinal can be the least $\omega_1$-strongly compact cardinal. Bagaria and Magidor showed that it is consistent, relative to a supercompact cardinal, that the least $\omega_1$-strongly compact cardinal is singular of cofinality greater than or equal to the least measurable cardinal \cite{BagariaMagidor:groupRadicals}. In this situation, a strong compactness cardinal for $\cL(Q^{WF})$ exists, but is greater than the least measurable cardinal. Note also that the canonical model $L[U]$ cannot have an $\omega_1$-strongly compact cardinal $\delta$. The only elementary embeddings in $L[U]$ are iterates of the ultrapower by the unique measure on the unique measurable cardinal $\kappa$ and only finite iterates have a target that is closed under $\omega$-sequences. But the existence of an $\omega_1$-strongly compact cardinal $\delta$ implies that there are elementary embeddings with targets closed under $\omega$-sequences mapping $\delta$ arbitrarily high in the ordinals. Thus, $L[U]$ does not have a strong compactness cardinal for $\cL(Q^{WF})$. Combining Corollary~\ref{cor:ULSWF} with the above results, we get the following corollary.

\begin{corollary} It is consistent that:
\begin{enumerate}
\item $\ULS(\cL(Q^{WF}))=\SULS(\cL(Q^{WF}))$ is the least strong compactness cardinal for $\cL(Q^{WF})$.
\item $\ULS(\cL(Q^{WF}))=\SULS(\cL(Q^{WF}))$ is smaller than the least strong compactness cardinal for $\cL(Q^{WF})$.
\item $\ULS(\cL(Q^{WF}))=\SULS(\cL(Q^{WF}))$ exists, but $\cL(Q^{WF})$ doesn't have a strong compactness cardinal.
\end{enumerate}

\end{corollary}
Thus, we have an example of a logic for which the $\ULS$ number is always equal to the strong $\ULS$ number, but consistently, it is possible that either the strong compactness cardinal doesn't exist, or it exists and is larger than the the strong $\ULS$ number, or it exists and is equal to the strong $\ULS$ number.

\section{Second-order logic}
Recall that a cardinal $\kappa$ is \emph{extendible} if for every $\alpha>\kappa$, there is $\beta>\kappa$ and an elementary embedding $j:V_\alpha\to V_\beta$ with $\crit(j)=\kappa$.\footnote{The assertion that $j(\kappa)>\alpha$ is often included in the definition of an extendible cardinal, but it can be shown that this leads to an equivalent notion (see \cite[23.15]{kanamori:higher}).} Recall that the least extendible cardinal is the least strong compactness cardinal for the logic $\cL^2$ and thus, $\SULS(\cL^2)$ is bounded by the least extendible cardinal by Proposition~\ref{prop:compactImpliesSULS}. In this section, we show that $\ULS(\cL^2)=\SULS(\cL^2)$ is precisely the least extendible cardinal.

\begin{theorem}\label{thm:ULSSecondOrderLogic}
If $\ULS(\cL^2)$ exists, then it is the least extendible cardinal.\footnote{Independently, this result was also obtained by Yair Hayut, and, also independently, by Will Boney and the second author, using different proofs, respectively.}
\end{theorem}
\begin{proof}
Let $\ULS(\cL^2)=\delta$. By Proposition~\ref{prop:compactImpliesSULS} and the leastness property of $\delta$, it suffices to show that there is an extendible cardinal ${\leq}\delta$. So suppose towards a contradiction that there is no extendible cardinal ${\leq}\delta$. Fix $\gamma\leq\delta$. If $\gamma$ is not $\gamma+1$-extendible, let $\alpha_\gamma=\gamma$ and let $\beta_\gamma=\gamma$. So suppose that $\gamma$ is at least $\gamma+1$-extendible. Given $\alpha\geq\gamma+1$ such that $\gamma$ is $\alpha$-extendible, let $\beta(\alpha)$ be the least $\beta$ such that there is an elementary embedding $j_\alpha:V_\alpha\to V_\beta$ with $\crit(j_\alpha)=\gamma$. Let $\alpha_\gamma$ be the supremum of ordinals $\alpha$ such that $\gamma$ is $\alpha$-extendible and let $\beta_\gamma$ be the supremum of the corresponding $\beta(\alpha)$. Let $\alpha^*$ be the supremum of the $\alpha_\gamma$ for $\gamma\leq\delta$ and let $\beta^*$ be the supremum of the $\beta_\gamma$. Let $\rho$ be the least cardinal above $|V_{\beta^*+1}|$. In particular, $V_{\beta^*+1}\in H_{\rho^+}$. Consider the structure $$\mathcal M=( H_{\rho^+}, \in,\delta,\alpha^*,\beta^*,\rho,\Tr),$$ where $\Tr$ is a truth predicate for $( H_{\rho^+},\in)$. Then $\mathcal M$ satisfies the sentence $\varphi$ in the second-order logic $\cL^2$, which is the conjunction of the sentences:
\begin{enumerate}
\item I am well-founded.
\item $\delta$ is a cardinal.
\item $\alpha^*$ is the supremum of the $\alpha_\gamma$ for $\gamma\leq\delta$, where $\alpha_\gamma$ are defined as above.
\item $\beta^*$ is the supremum of the $\beta_\gamma$ for $\gamma\leq\delta$, where $\beta_\gamma$ are defined as above.
\item $V_{\beta^*+1}$ exists.
\item $\Tr$ is a truth predicate for $( H_{\rho^+},\in)$.
\item $\rho$ is the largest cardinal.
\item $\rho$ is the least cardinal above $|V_{\beta^*+1}|$.
\end{enumerate}
Since $\delta=\ULS(\cL^2)$, there is a model $$\mathcal N=(N,{\mathsf E},\overline \delta,\overline\alpha^*,\overline\beta^*,\overline\rho,\overline \Tr\ra\models_{\cL^2}\varphi,$$ of cardinality much larger than $\rho$ and having $\mathcal M$ as a substructure. Since ${\mathsf E}$ is well-founded, we can assume, by collapsing, that ${\mathsf E}=\in$ and there an embedding $$j:H_{\rho^+}\to N$$ such that $j(\delta)=\overline\delta$, $j(\alpha^*)=\overline\alpha^*$, $j(\beta^*)=\overline\beta^*$, and $j(\rho)=\overline\rho$. Since our structures contain a truth predicate, the embedding $j$ is elementary. The model $N$ must contain the correct $V_{\overline\beta^*+1}$ because of our usage of second-order logic. Since \hbox{$|N|\gg \rho$}, it follows that $\overline\rho\gg\rho$. Thus, since $\overline\rho$ is the least cardinal above $|V_{\overline\beta^*+1}|$ in $N$, it must be the case that $\overline\beta^*\gg\beta^*$.

Let's argue that $\delta<\overline\delta$. If $\delta=\overline\delta$, then since $N$ has the correct $V_{\beta^*+1}$, it would compute $\alpha^*$ and $\beta^*$ the same as $V$, but this would contradict its definition of $\overline\beta^*$. Thus, $j$ has a critical point $\kappa\leq\delta$. Let $$j:V_{\alpha^*+1}\to V_{\overline\alpha^*+1}$$ be the restriction of $j$. Then $j$ witnesses that $\kappa$ is $\alpha^*+1$-extendible, contradicting that $\alpha^*$ was the supremum of extendibility for $\gamma\leq\delta$.
\end{proof}
\begin{corollary}
The following are equivalent for a cardinal $\kappa$.
\begin{enumerate}
\item $\kappa$ is the least extendible cardinal.
\item $\kappa$ is the least strong compactness cardinal for $\cL^2$.
\item $\kappa=\SULS(\cL^2)$.
\item $\kappa=\ULS(\cL^2)$.
\end{enumerate}	
\end{corollary}
Thus, we have an example of a logic stronger than first-order logic for which the $\ULS$ number is same as the strong $\ULS$ number and the same as the least strong compactness cardinal.
\section{Sort logics}
In this section we show that for every $n<\omega$, $\ULS(\cL^{s,n})=\SULS(\cL^{s,n})$ is the least $C^{(n)}$-extendible cardinal.

Recall that for every $n<\omega$, $C^{(n)}=\{\alpha\in\Ord\mid V_\alpha\prec_{\Sigma_n} V\}$ is the class of $\Sigma_n$-reflecting cardinals.
Recall next that a cardinal $\kappa$ is \emph{$C^{(n)}$-extendible} if for every $\alpha>\kappa$ in $C^{(n)}$, there is $\beta\in C^{(n)}$ and an elementary embedding $j:V_\alpha\to V_\beta$ with $\crit(j)=\kappa$.\footnote{Again, the assertion that $j(\kappa)>\alpha$ is often included in the definition, but it follows without loss of generality.} Boney showed in \cite{Boney:LargeCardinalLogics} that the least $C^{(n)}$-extendible cardinal is the least strong compactness cardinal for the sort logic $\cL^{s,n}$.

By exactly mimicking the proof of Theorem~\ref{thm:ULSSecondOrderLogic} and using that the assertion $V_\alpha\prec_{\Sigma_n}V$ is expressible in the logic $\cL^{s,n}$, we get:
\begin{theorem}
Fix $n<\omega$. If $\ULS(\cL^{s,n})$ exists, then it is the least $C^{(n)}$-extendible cardinal.\footnote{Independently, this results was also obtained by Will Boney and the second author, using a different proof.}
\end{theorem}
\begin{corollary}
The following are equivalent for a cardinal $\kappa$ and $n<\omega$.
\begin{enumerate}
\item $\kappa$ is the least $C^{(n)}$-extendible cardinal.
\item $\kappa$ is the least strong compactness cardinal for $\cL^{s,n}$.
\item $\kappa=\SULS(\cL^{s,n})$.
\item $\kappa=\ULS(\cL^{s,n})$.
\end{enumerate}	

\end{corollary}

Recall that Makowsky showed that Vop\v{e}nka's Principle is equivalent to the assertion that every logic has a strong compactness cardinal. Bagaria showed that Vop\v{e}nka's Principle is equivalent to the assertion that for every $n<\omega$, there is a $C^{(n)}$-extendible cardinal \cite{Bagaria:VopenkaCnExtendibility}. It follows that:
\begin{theorem}
Vop\v{e}nka's Principle holds if and only if for every logic $\sL$, $\ULS(\sL)$ exists.
\end{theorem}
\section{The logics $\cL_{\eta,\eta}$}
In this section, we consider infinitary logics $\cL_{\eta,\eta}$ with $\eta$ an uncountable regular cardinal. We show that $\SULS(\cL_{\kappa,\kappa})=\kappa$ if and only if $\kappa$ is a tall cardinal, and, more generally, that the existence of $\ULS$ and $\SULS$ numbers for these logics is related to the existence of tall-like cardinals.

Recall that for every $\alpha<\eta$, there is a sentence $\psi_\alpha$ in $\cL_{\eta,\eta}$, which over a transitive model of set theory $N$, expresses closure under $\alpha$-sequences, $N^{\alpha}\subseteq N$.

\begin{proposition}\label{prop:ulsLetaetageqeta}
If $\ULS(\cL_{\eta,\eta})$ exists, then $\ULS(\cL_{\eta,\eta})\geq\eta$.
\end{proposition}
\begin{proof}
Recall that for every ordinal $\xi<\eta$ and definable binary relation $\psi(y,x)$, there is a corresponding formula $\varphi^\xi_{\psi}(x)$ expressing that the relation given by $\psi$ when restricted to the predecessors of $x$ is isomorphic to the well-order $(\xi,\in)$.  Suppose that $\ULS(\cL_{\eta,\eta})=\delta<\eta$. Consider the model $\mathcal M=(\delta,\in)$ and let $\psi(y,x):=y\in x$. Then $\mathcal M$ satisfies that $\in$ is a well-order and there is no $x$ satisfying $\varphi^\delta_\psi(x)$.  But any larger well-order will have a witness for $\varphi^{\delta}_\psi(x)$.
\end{proof}

Recall that a cardinal $\kappa$ is $\theta$-\emph{tall}, for $\theta>\kappa$, if there exists an elementary embedding $j:V\to M$ with $\crit(j)=\kappa$, $j(\kappa)>\theta$ and $M^\kappa\subseteq M$, and a cardinal $\kappa$ is \emph{tall} if it is $\theta$-tall for every $\theta>\kappa$. The difference between a tallness embedding and an iterated measurability embedding is the closure on the target model $M$. More generally, a cardinal $\kappa$ is $\theta$-\emph{tall with closure $\lambda$} ($\omega\leq \lambda\leq\kappa$) if  there exists an elementary embedding $j:V\to M$ with $\crit(j)=\kappa$, $j(\kappa)>\theta$ and $M^\lambda\subseteq M$, and a cardinal $\kappa$ is \emph{tall with closure $\lambda$} if it is $\theta$-tall with closure $\lambda$ for every $\theta>\kappa$. If the target model $M$ has closure $M^{{<}\lambda}\subseteq M$, then we say that $\kappa$ is ($\theta$-)\emph{tall with closure ${<}\lambda$}. All these cardinals were introduced in \cite{Hamkins:TallCardinals}.

\begin{proposition}[Theorem 5.1, Hamkins \cite{Hamkins:TallCardinals}]\label{prop:tall}
If a cardinal $\kappa$ is tall with closure ${<}\kappa$, then $\kappa$ is tall.
\end{proposition}
The following lemma can be shown using similar (in fact, easier) arguments as Corollary \ref{CharPushingUp1}, so we will omit the proof here.
\begin{lemma}\label{lem:setEmbeddingsTall}
A cardinal $\kappa$ is $\theta$-tall with closure ${<}\kappa$ if and only if for some $\alpha > \kappa$ there is an elementary embedding $h: V_\alpha \rightarrow N$ with $\crit(h) = \kappa$, $h(\kappa) > \theta$ and such that $N$ has, for all $\lambda < \kappa$, all functions $\lambda \rightarrow [j(\kappa)]^{< \omega}$.
		
\end{lemma}
We can therefore witness tallness by embeddings between set-sized structures. We proceed to consider the relationship of $\SULS(\cL_{\eta,\eta})$ with tall cardinals.
\begin{theorem}\label{th:tallImpliesSULS}
If there is a  tall cardinal $\kappa \geq \eta$ with closure ${<}\eta$, then $\SULS(\cL_{\eta,\eta})$ exists and is at most $\kappa$. In particular, if $\kappa$ is tall, then $$\SULS(\cL_{\kappa,\kappa})=\ULS(\cL_{\kappa,\kappa})=\kappa.$$
\end{theorem}
\begin{proof}
Suppose that $N$ is a $\tau$-structure of size $\gamma\geq\kappa$. Let $\overline\gamma >\gamma$ be a cardinal. Let $j:V\to M$ be an elementary embedding with $\crit(j)=\kappa$, $j(\kappa)>\overline\gamma^+$ and $M^{{<}\eta}\subseteq M$. We now argue as in the proof of Theorem~\ref{th:measImpliesSULSWF}. Consider the $j(\tau)$-structure $j(N)\in M$. Since $j\image\tau\subseteq j(\tau)$, $j(N)$ is a $\tau$-structure modulo the renaming which takes $\tau$ to $j\image\tau$. The renaming gives rise to the associated bijection between formulas in $\tau$ and $j\image\tau$, so specifically, let it send $\varphi$ to $\overline\varphi$. Consider the $\tau$-substructure $\overline N=j\image N\subseteq j(N)$. The bijection $j$ witnesses that $N\cong \overline N$ as $\tau$-structures. Next, let's show that $\overline N\prec_{\cL_{\eta,\eta}} j(N)$. Suppose that $\overline N\models_{\cL_{\eta,\eta}}\varphi(j(a))$. Via the isomorphism, it follows that $N\models_{\cL_{\eta,\eta}}\varphi(a)$. But then by elementarity of $j$, $M$ satisfies that $j(N)\models_{\cL_{\eta,\eta}}\overline\varphi(j(a))$ because $\cL_{\eta,\eta}$-formulas are fixed by $j$ whose critical point is $\kappa\geq\eta$. Since $M$ is closed under ${<}\eta$-sequences, it is correct about $\cL_{\eta,\eta}$ satisfaction, and so it is actually the case that $j(N)\models_{\cL_{\eta,\eta}}\overline\varphi(j(a))$ in $j(\tau)$. Hence $j(N)\models_{\cL_{\eta,\eta}}\varphi(j(a))$ modulo the renaming. Finally, since $|N|\geq\gamma\geq\kappa$, $M$ satisfies that $|j(N)|\geq j(\kappa)>\overline\gamma^+$, but then in $V$, $|j(N)|\geq\overline\gamma^+>\overline\gamma$ as desired.

We just argued that $\SULS(\cL_{\eta,\eta})\leq\kappa$. By Proposition~\ref{prop:ulsLetaetageqeta}, $\ULS(\cL_{\eta,\eta})\geq\eta$. Thus, in particular, if $\kappa$ is tall, then $\SULS(\cL_{\kappa,\kappa})=\ULS(\cL_{\kappa,\kappa})=\kappa$.
\end{proof}
Recall that $\cL(Q^{WF})\subseteq \cL_{\omega_1,\omega_1}$. In particular, a ULST number of $\cL_{\omega_1, \omega_1}$ must be greater or equal than the least measurable cardinal. Since strongly compact cardinals are tall \cite[Theorem 2.11]{Hamkins:TallCardinals} and it is consistent that the least measurable is the least strongly compact, we get the following:
\begin{theorem}
It is consistent that the following are equivalent for a cardinal $\kappa$:
\begin{enumerate}
\item $\kappa$ is the least measurable cardinal.
\item $\kappa$ is the least strong compactness cardinal for $\cL(Q^{WF})$.
\item $\kappa$ is the least cardinal which is a strong compactness cardinal for $\cL_{\kappa,\kappa}$.
\item $\kappa=\SULS(\cL(Q^{WF}))$.
\item $\kappa=\SULS(\cL_{\eta,\eta})$ for all uncountable regular $\eta\leq\kappa$.
\item $\kappa=\ULS(\cL(Q^{WF}))$
\item $\kappa=\ULS(\cL_{\eta,\eta})$ for all uncountable regular $\eta\leq\kappa$.
\end{enumerate}
\end{theorem}
\begin{theorem}\label{th:SULSInfinLogicItselfImpliesTall}
If $\SULS(\cL_{\kappa,\kappa})=\kappa$, then $\kappa$ is tall.
\end{theorem}
\begin{proof}
By Proposition~\ref{prop:tall}, it suffices to show that $\kappa$ is tall with closure ${<}\kappa$. We will show that for every cardinal $\theta>\kappa$, there is an elementary embedding \hbox{$j_\theta:V_{\kappa+1}\to N_\theta$} with $\crit(j_\theta)=\kappa$, $j_\theta(\kappa)>\theta$ and such that $N_\theta$ has all functions $\lambda \rightarrow [j(\kappa)]^{<\omega}$ for all $\lambda < \kappa$. This suffices by Lemma~\ref{lem:setEmbeddingsTall}. Consider the structure $$\mathcal M=(H_{\kappa^+},\in,\kappa).$$ Because $\cL_{\kappa,\kappa}$ can define well-foundedness, using an argument analogous to the proof of Theorem~\ref{th:ULSWF}, we can see that there is an $\cL_{\kappa,\kappa}$-elementary embedding $h:H_{\kappa^+}\to N$ with $\crit(h)\leq\kappa$. But since every $\xi<\kappa$ is definable in $\cL_{\kappa,\kappa}$, it follows that $\crit(h)=\kappa$. Thus, $\kappa$ is at least measurable, and in particular, $V_{\kappa +1 } \subseteq H_{\kappa^+}$. Note that $\mathcal M\models_{\cL_{\kappa,\kappa}}\psi_\alpha$ (witnessing that $H_{\kappa^+}^{\alpha}\subseteq H_{\kappa^+}$) for every $\alpha<\kappa$. Fix $\theta>\kappa$. Since $\kappa=\SULS(\cL_{\kappa,\kappa})$, there is a model $\mathcal N=(N,\mathsf E,\overline\kappa)$ of size larger than the smallest $\beth$-fixed point $\beth_\rho = \rho > \theta$ and such that $\mathcal M\prec_{\cL_{\kappa,\kappa}}\mathcal N$. It follows that $\mathsf E$ is well-founded, so we can assume, by collapsing, that ${\mathsf E}=\in$, $N$ is transitive and $$j:M \to N$$ is an $\cL_{\kappa,\kappa}$-elementary embedding. Then $\rho \in N$ and since $\bar\kappa$ is the largest cardinal of $N$, we get $\kappa<\theta< \rho \leq \bar\kappa=j(\kappa)$. Since the embedding $j$ is $\cL_{\kappa,\kappa}$-elementary, and every ordinal $\xi<\kappa$ is $\cL_{\kappa,\kappa}$-definable, it follows that $\crit(j)=\kappa$. We also have, for every $\alpha<\kappa$, $N\models_{\cL_{\kappa,\kappa}}\psi_{\alpha}$, and so   $N^{{<}\kappa}\subseteq N$. Further, $N$ believes that $V_{j(\kappa) + 1}^N$ has all functions $\lambda \rightarrow [j(\kappa)]^{< \omega}$ and, by its own closure, it is correct about this. Thus, the restriction $j:V_{\kappa+1}\to V_{j(\kappa) +1}^N$ has all the required properties.

\end{proof}
\begin{corollary}
A cardinal $\kappa$ is tall if and only if $\ULS(\cL_{\kappa,\kappa})=\SULS(\cL_{\kappa,\kappa})=\kappa$.
\end{corollary}
Since consistency-wise strongly compact cardinals are much stronger than tall cardinals (which are equiconsistent with strong cardinals (see \cite{Hamkins:TallCardinals})), it is consistent to have a tall cardinal that is not strongly compact.
\begin{corollary}
It is consistent that $\ULS(\cL_{\kappa,\kappa})=\SULS(\cL_{\kappa,\kappa})=\kappa$, but $\kappa$ is not a strong compactness cardinal for $\cL_{\kappa,\kappa}$.
\end{corollary}

Next, we introduce a version of tall cardinals $\kappa$, where the defining embeddings, instead of mapping $\kappa$ as high as desired, map some fixed ordinal $\delta\geq\kappa$ as high as desired.
\begin{definition}
A cardinal $\kappa\leq\delta$ is \emph{$\theta$-tall pushing up $\delta$ with closure $\lambda\leq\kappa$} if there is an elementary embedding $j:V\to M$ with $\crit(j)=\kappa$, $M^{\lambda}\subseteq M$, and $j(\delta)>\theta$. A cardinal $\kappa\leq\delta$ is \emph{tall pushing up $\delta$ with closure $\lambda$} if it is $\theta$-tall pushing up $\delta$ with closure $\lambda$ for all $\theta>\kappa$. If the target model $M$ has closure $M^{{<}\lambda}\subseteq M$, then we say that $\kappa$ is ($\theta$-)\emph{tall pushing up $\delta$ with closure ${<}\lambda$}.
\end{definition}

Observe that a $\theta$-tall cardinal $\kappa$ with closure $\lambda$ is $\theta$-tall pushing up $\kappa$ with closure $\lambda$. But in our more general definition it might be a larger ordinal than $\kappa$ that gets mapped beyond $\theta$.

We would like to thank Joel David Hamkins for pointing out the following result separating tall cardinals from tall cardinal pushing up some $\delta$.
\begin{proposition}[Hamkins]
It is consistent that there is a cardinal $\kappa$ which is not tall, but for which there is an ordinal $\delta>\kappa$ such that $\kappa$ is tall pushing up $\delta$.
\end{proposition}
\begin{proof}
Suppose we have a model in which $\kappa$ is measurable but not tall and $\delta>\kappa$ is tall. Let's argue that $\kappa$ is tall pushing up $\delta$. Fix an ordinal $\theta$ and let $j:V\to M$ be an elementary embedding with $M^\delta\subseteq M$, $\crit(j)=\delta$, and $j(\delta)>\theta$, witnessing the $\theta$-tallness of $\delta$. Let $h:V\to N$ be the ultrapower embedding by a $\kappa$-complete ultrafilter on $\kappa$, so that we have $\crit(h)=\kappa$ and $N^\kappa\subseteq N$. Let $j:N\to \overline N$ be the restriction of $j$ to $N$. Since $N^{\kappa}\subseteq N$, by elementarity, we get that $\overline N^{j(\kappa)}\subseteq \overline N$ in $M$. By the closure of $M$, we get that $\overline N^\kappa\subseteq \overline N$. The composition $j\circ h:V\to \overline N$ now witnesses that $\kappa$ is $\theta$-tall pushing up $\delta$.
\end{proof}
\begin{question}
Is the existence of a tall cardinal $\kappa$ pushing up some $\delta>\kappa$ equiconsistent with a tall cardinal?
\end{question}
We want to show that $\kappa$ being tall pushing up $\delta$ with closure ${<}\eta$ is witnessed by extenders and, thus, by set-sized embeddings. We assume that $\delta$ is a strong limit cardinal, $\eta\leq\kappa$ is regular, and $\theta^\omega > \theta$ (e.g., $\text{cof}(\theta)=\omega$). Suppose we have an embedding $j : V \rightarrow M$ with $\text{crit}(j) = \kappa$, $j(\delta) > \theta$ and such that $M$ is ${<}\eta$-closed witnessing that $\kappa$ is $\theta$-tall pushing up $\delta$ with closure ${<}\eta$. By the closure of $M$, for any $\beta < \eta$ we get that ${}^\beta \theta = ({}^\beta \theta)^M$ and thus $\theta^\beta \leq (\theta^\beta)^M$. Further, because $M$ believes that $j(\delta)$ is a strong limit cardinal and $\beta < \theta < j(\delta)$, we have that $\theta^\beta \leq (\theta^\beta)^M < j(\delta)$. Let $\gamma = \text{sup}\{(\theta^\beta)^M \colon \beta < \eta \}$. Notice that by our remarks we have that $\gamma \leq j(\delta)$ and, thus, it makes sense to derive an extender $E$ from $j$ by letting for $a \in [\gamma]^{< \omega}$ and $X \subseteq [\delta]^{|a|}$:
\[
X \in E_a \text{ iff } a \in j(X).
\]
Let $M_a$ be the ultrapower of $V$ by $E_a$ and let $M_E$ be the direct limit of the $M_a$. Standard results (see \cite[26.1 and 26.2]{kanamori:higher}) imply that there are canonical embeddings $j_E: V \rightarrow M_E$ and $k: M_E \rightarrow M$ such that $j = k \circ j_E$ and with $\text{crit}(j_E) = \kappa$, $j_E(\delta) \geq \gamma$ and $\text{crit}(k) \geq \gamma$, where $k$ is the inverse transitive collapse. Further
\[
M_E = \{j_E(f)(a) \mid a \in [\gamma]^{< \omega}, f:[\delta]^{|a|} \rightarrow V \}
\]
and
\[
\text{ran}(k) = \{j(f)(a) \mid a \in [\gamma]^{< \omega}, f: [\delta]^{|a|} \rightarrow V \}.
\]
Notice that $\text{ran}(k)$ is an elementary substructure of $M$. We want to see that $j_E:V \rightarrow M_E$ witnesses that $\kappa$ is $\theta$-tall pushing up $\delta$ with closure ${<}\eta$. It remains to check the closure of $M_E$. So let $\nu < \eta$ and fix $$\{j_E(f_\alpha)(a_\alpha)\mid \alpha < \nu\} \subseteq M_E$$  with each $a_\alpha \in [\gamma]^{< \omega}$ and $f_\alpha: [\delta]^{|a_\alpha|} \rightarrow V$. We have that $$j(\{f_\alpha\mid\alpha<\nu\})=\{j(f_\alpha)\mid\alpha<\nu\}\in M_E.$$ So if we can show that $\{a_\alpha\mid \alpha < \nu\} \in M_E$, we are done as then the pointwise evaluation of $\{j(f_\alpha)\mid \alpha < \nu\}$ will also be in $M_E$. Because $$\text{crit}(k) \geq \gamma \geq (\theta^\nu)^M \geq \theta^\nu$$ we know that $\gamma \subseteq \text{ran} (k)$ and as we assumed $\theta^\nu \geq \theta^\omega > \theta > \eta > \nu$, we have $\theta, \eta, \nu \in \text{ran}(k)$. Because $M$ is closed under $\nu$-sequences, ${}^\nu([\gamma]^{< \omega}) \subseteq M$, and so $M$'s version of ${}^\nu([\gamma]^{< \omega})$ is the real ${}^\nu([\gamma]^{< \omega})$. Now notice that $\gamma$ is definable from $\eta$ and $\theta$ in $M$ and, thus, by elementarity, $\gamma$ must be in $\text{ran}(k)$. Similarly, ${}^\nu([\gamma]^{< \omega})$ is definable from $\nu$ and $\gamma$ and so again by elementarity we get that ${}^\nu([\gamma]^{< \omega})$ has to be in $\text{ran}(k)$ as well. Further, $M$ believes that there is an enumeration $g$ of ${}^\nu([\gamma]^{< \omega})$ and g has domain $(\gamma^\nu)^M$. Thus, $\text{ran}(k)$ also has such an enumeration, say $h$. We claim that $(\gamma^\nu)^M = \gamma$. Notice that by definition of $\gamma$ and regularity of $\eta$, we either have that $\gamma=(\theta^\beta)^M$ for some $\beta<\eta$ or $\text{cof}(\gamma)=\eta$.  In the first case, in $M$, we have $\gamma^\nu=(\theta^\beta)^\nu=\theta^\beta=\gamma$. So let's assume that $\text{cof}(\gamma)=\eta$. Thus, because $\nu < \eta$, if $f: \nu \rightarrow \gamma$ is a function in $M$, then $\text{ran}(f) \subseteq (\theta^\beta)^M$ for some $\beta < \eta$, and we have $$(\gamma^\nu)^M \leq \sup \{ ((\theta^\beta)^\nu)^M \mid \beta < \eta \} = \text{sup}\{(\theta^\beta)^M \mid \beta < \eta \} = \gamma.$$ Thus $(\gamma^\nu)^M = \gamma$. Therefore $h$ is an enumeration of ${}^\nu ([\gamma]^{< \omega})$ with domain $\gamma \subseteq \text{ran}(k)$ and we get that the evaluation $h(\alpha)$ at any $\alpha < \gamma \subseteq \text{ran}(k)$ is in $\text{ran}(k)$, and so ${}^\nu([\gamma]^{< \omega}) \subseteq \text{ran}(k)$. In particular, $\{a_\alpha\mid \alpha < \nu\} \in \text{ran}(k)$. Fix $f\in {}^\nu([\gamma]^{< \omega}) \subseteq \text{ran}(k)$. Our argument above also shows that $f:\nu\to[\theta^\beta]^{<\omega}$ for some $\beta<\eta$. Since $k$ fixes $\beta$ and $\theta$, it follows that $k(\theta^\beta)=\theta^\beta$. If $\theta^\beta<\gamma$, then $k(f)=f$ and if $\theta^\beta=\gamma$, then $\crit(k)>\gamma$, and so also, $k(f)=f$.  It follows that $$\{a_\alpha\mid \alpha < \nu\} = k^{-1}(\{a_\alpha\mid \alpha < \nu\}) \in M_E.$$

Notice that in the above argument, the fact that $V_{j(\delta)+1}^M$ has all functions\break $f:\nu \rightarrow [\gamma]^{< \omega}$ is sufficient to show that $M_E$ is ${<}\eta$-closed. Thus we can conclude that $\theta$-tallness pushing up $\delta$ with closure ${<}\eta$ is, under the above conditions, already witnessed by set sized embeddings:
\begin{lemma}
Suppose that $\delta$ is a strong limit cardinal, $\eta\leq\kappa$ is regular and $\theta>\delta$ is such that $\theta^\omega > \theta$. Assume that for some $\alpha > \delta$ there is an elementary embedding $j: V_{\alpha} \rightarrow N$ with $\text{crit}(j) = \kappa$ and $j(\delta) > \theta$. Further let $\gamma = \text{sup}\{(\theta^\beta)^N \mid \beta < \eta \}$ and assume that $N$ has, for all $\nu < \eta$, all functions $\nu \rightarrow [\gamma]^{< \omega}$. Let $E$ be the extender with seed set $[\gamma]^{< \omega}$ derived from $j$ consisting of ultrafilters on $[\delta]^n$. Then the canonical embedding $j_E: V \rightarrow M_E$ into the extender power of $V$ by $E$ witnesses that $\kappa$ is $\theta$-tall pushing up $\delta$ with closure ${<}\eta$.
\end{lemma}
Because $j_E$ restricts to an embedding $j_E : V_{\delta + 1} \rightarrow V_{j(\delta) + 1}^{M_E}$, we therefore get:
\begin{corollary}\label{CharPushingUp1}
	Suppose that $\delta$ is a strong limit cardinal, $\eta\leq\kappa$ is regular and $\theta>\delta$ is such that $\theta^\omega > \theta$. Then $\kappa$ is $\theta$-tall pushing up $\delta$ with closure ${<}\eta$ if and only if for some $\alpha > \delta$ there is an elementary embedding $h: V_\alpha \rightarrow N$ with $N$ transitive, $\text{crit}(h) = \kappa$, $h(\delta) > \theta$ and such that for all $\nu < \eta$, $N$ has all functions $\nu \rightarrow [\gamma]^{< \omega}$, where $\gamma = \sup\{(\theta^\beta)^N \mid \beta < \eta \}$.
\end{corollary}
The following corollary can be shown using similar (in fact, easier) arguments as Corollary \ref{CharPushingUp1}
\begin{corollary}\label{CharPushingUp2}
Suppose that $\delta$ is a strong limit cardinal, $\lambda<\kappa$ and $\theta>\delta$ is such that $\theta^\omega > \theta$. Then $\kappa$ is $\theta$-tall pushing up $\delta$ with closure $\lambda$ if and only if for some $\alpha > \delta$ there is an elementary embedding $h: V_\alpha \rightarrow N$ with $N$ transitive, $\text{crit}(h) = \kappa$, $h(\delta) > \theta$ and such that $N$ has all functions $\lambda \rightarrow [(\theta^\lambda)^N]^{< \omega}$.
\end{corollary}

We will use these results to get more general versions of Theorems \ref{th:tallImpliesSULS} and \ref{th:SULSInfinLogicItselfImpliesTall} as well as results about the $\ULS$ numbers.
\begin{theorem}\label{th:tallPushingDeltaImliesSULS}
If there is a tall cardinal $\kappa$ pushing up $\delta$ with closure ${<}\eta$ for some $\delta\geq\kappa$ and regular $\eta\leq\kappa$, then $\SULS(\cL_{\eta,\eta})$ exists and is at most $\delta$.
\end{theorem}
The proof is completely analogous to the proof of Theorem~\ref{th:tallImpliesSULS} and indeed, Theorem~\ref{th:tallImpliesSULS} can now be derived as a corollary.
\begin{theorem}\label{th:SULSeta}
	If $\SULS(\mathbb L_{\eta, \eta}) = \delta$, then there is a cardinal $\eta \leq \kappa \leq \delta$ that is tall pushing up $\delta$ with closure ${< }\eta$.
	\begin{proof}
Let $\rho$ be the first $\beth$-fixed point above $\delta$. We first argue that for any $\beth$-fixed point $\theta > \delta$ there is an elementary embedding $j_\theta: V_{\rho} \rightarrow N_\theta$ with $\eta \leq \text{crit}(j_\theta) \leq \delta$, $j_\theta(\delta) > \theta$ and such that $N_\theta$ has all functions $f:\nu \rightarrow [\alpha]^{< \omega}$ for all $\nu < \eta$ and all ordinals $\alpha \in N_\theta$. We consider the structure $\mathcal M = (V_{\rho}, \in, \delta)$. Since $\text{SULST}(\mathbb L_{\eta, \eta}) = \delta$, there is a model $\mathcal N=(N,\mathsf E,\bar\delta)$, with $|N| \gg \theta$, and such that $\mathcal M\prec_{\cL_{\eta,\eta}}\mathcal N$. Because $\mathcal M$ satisfies the $\mathbb L_{\eta, \eta}$-sentence asserting well-foundedness, we can assume that $N$ is transitive, $\mathsf E = \in$ and $j: M \rightarrow N$ is $\mathbb L_{\eta, \eta}$-elementary with $j(\delta) = \bar \delta$. Further, $\mathcal M$ satisfies, for every $\nu < \eta$, the sentence $\chi_\nu$ of $\mathcal L_{\eta, \eta}$, truthfully asserting that $V_\rho$ has all functions $\lambda \rightarrow [\alpha]^{< \omega}$ for any ordinal $\alpha$ in $V_\rho$.
Because $|N| \gg \theta$, we get that $\theta \in N$. By elementarity, $\mathcal N$ believes that there is no $\beth$-fixed point above $\bar \delta$, so since $\mathcal N$ sees that $\theta$ is a $\beth$-fixed point, it follows that $\bar\delta > \theta >\delta$. In particular, $\crit(j)  \leq \delta$. Since every ordinal $\xi<\eta$ is definable in the logic $\cL_{\eta,\eta}$, we must have $\eta\leq\crit(j)$. Because $\mathcal N$ satisfies the sentences $\chi_\nu$ for all $\nu < \eta$, it has all the required functions. So $j$ is how we promised.
		
Because there are unboundedly many $\theta > \delta$ but boundedly many $\kappa \leq \delta$, we can fix a single $\eta \leq \kappa \leq \delta$ such that for any $\theta>\delta$, there is an elementary embedding $j_\theta: V_{\rho} \rightarrow N_\theta$ with $\crit(j_\theta)=\kappa$, $j_\theta(\delta) > \theta$ and such that $N_\theta$ has all functions $\nu \rightarrow [\alpha]^{< \omega}$ for all $\nu < \eta$ and all ordinals $\alpha \in N_\theta$. Let $\kappa\leq\delta^*\leq \delta$ be the least cardinal such that for any $\theta>\delta^*$, there is an elementary embedding $j_\theta: V_{\rho^*} \rightarrow N_\theta$, where $\rho^*$ is the least $\beth$-fixed point above $\delta^*$, with $\crit(j_\theta)=\kappa$, $j_\theta(\delta^*) > \theta$ and such that $N_\theta$ has all functions $\nu \rightarrow [\alpha]^{< \omega}$ for all $\nu < \eta$ and all ordinals $\alpha \in N_\theta$. Let's argue that $\delta^*$ is a strong limit. If $\delta^*$ is not a strong limit, then there is $\gamma < \delta^*$ with $\rho^* > 2^\gamma \geq \delta^*$. Note that $\rho^*$ is the least $\beth$-fixed point above $\gamma$. Consider any strong limit $\theta> \rho^*$. By assumption, there is an elementary embedding $j_\theta: V_{\rho^*}\rightarrow N_\theta$ with $\crit(j_\theta)= \kappa^*$ and $j_\theta(\delta^*) > \theta$. Now because $2^\gamma \geq \delta^*$, by elementarity $$N_\theta \models 2^{j(\gamma)} \geq j(\delta^*) > \theta.$$ But then because $\theta$ is a strong limit, also $N_\theta \models j(\gamma) \geq \theta$. Because this works for any $\theta$, this is a contradiction to the minimality of $\delta^*$, verifying that $\delta^*$ is a strong limit cardinal. Since the target models $N_\theta$ have all the functions required in Corollary~\ref{CharPushingUp1}, the embeddings we produced verify that $\kappa^*$ is tall pushing up $\delta^*$, so in particular also pushing up $\delta\geq \delta^*$, with closure $< \eta$.
	\end{proof}
\end{theorem}
\begin{corollary}
	$\SULS(\cL_{\eta, \eta}) = \delta$ if and only if $\delta$ is the smallest cardinal $\geq \eta$ such that there is a tall cardinal $\eta\leq\kappa\leq\delta$ pushing up $\delta$ with closure $< \eta$.
\end{corollary}

Observe that in the canonical model $L[U]$, there are no tall cardinals $\gamma\leq\delta$ pushing up $\delta$ with closure $\omega$ for the same reason that there are no $\omega_1$-strongly compact cardinals. Thus, in particular, it is consistent that there is a measurable cardinal, but there is no pair $\gamma\leq\delta$ such that $\gamma$ is a tall cardinal pushing up $\delta$ with closure $\omega$. In particular, by Theorem~\ref{th:SULSeta}, if $\SULS(\cL_{\omega_1,\omega_1})$ exists, we must already have a pair $\kappa\leq\delta$ such that $\kappa$ is a tall cardinal pushing up $\delta$ with closure $\omega$. So having an SULST number for $\cL_{\omega_1,\omega_1}$ is stronger than having an SULST number for $\cL(Q^{WF})$.

To consider situations where $\ULS(\cL_{\eta, \eta}) = \eta$, we introduce the following concept.
	
\begin{definition}
	A cardinal $\delta$ is \emph{supreme for tallness} if for all $\lambda<\delta$ and ordinals $\theta$, there is a cardinal $\lambda<\kappa\leq\delta$ that is $\theta$-tall pushing up $\delta$ with closure $\lambda$.
\end{definition}
Observe that a cardinal $\delta$ is supreme for tallness if and only if for every $\lambda<\delta$, there is a cardinal $\lambda < \kappa\leq\delta$ that is tall pushing up $\delta$ with closure $\lambda$. This follows because there are proper class many $\theta$ and the cardinals $\kappa$ are bounded by $\delta$. Observe also that a tall cardinal is trivially supreme for tallness. A non-tall cardinal that is a limit of tall cardinals is also supreme for tallness. Thus, a supreme for tallness cardinal can be singular. On the other hand, a regular supreme for tallness cardinal is inaccessible because it is a limit of measurable cardinals. But we show below that it need not be weakly compact (Theorem~\ref{th:ULSnotSULS}).

\begin{theorem}\label{th:ULSLeta}
	If $\delta$ is supreme for tallness, then for every regular $\eta\leq\delta$, $\ULS(\cL_{\eta,\eta})$ exists and is at most $\delta$. In particular, if $\delta$ is regular, then $\ULS(\cL_{\delta,\delta})=\delta$.
\end{theorem}
\begin{proof}
	Suppose that $N$ is a $\tau$-structure of size $\gamma\geq\delta$, $\eta\leq\delta$, and $\varphi$ is a sentence in $\cL_{\eta,\eta}(\tau)$ such that $N\models_{\cL_{\eta,\eta}}\varphi$. Since $\eta$ is regular, there is $\lambda<\eta$ such that the length of all conjunctions and quantifiers in $\varphi$ is at most $\lambda$. Let $\overline\gamma>\gamma$ be a cardinal. By our assumption there exists a cardinal $\lambda<\kappa\leq\delta$ such that $\kappa$ is $\overline\gamma^+$-tall pushing up $\delta$ with closure $\lambda$. Let $j:V\to M$ be an elementary embedding with $\crit(j)=\kappa$, $j(\delta)>\overline\gamma^+$, and $M^\lambda\subseteq M$. Consider the $j(\tau)$-structure $j(N)$. Since $j\image\tau\subseteq\tau$, $j(N)$ is a $\tau$-structure modulo the renaming which takes $\tau$ to $j\image\tau$. The renaming gives rise to the associated bijection between formulas in $\tau$ and $j\image\tau$, so specifically, let it send $\psi$ to $\overline\psi$.
Consider the $\tau$-substructure $\overline N=j\image N\subseteq j(N)$. Clearly, the bijection $j$ witnesses that $N\cong \overline N$ as $\tau$-structures. Thus, via the isomorphism, we have that $N$ is a $\tau$-substructure of $j(N)$. By elementarity, $M$ satisfies that $j(N)\models_{\cL_{j(\eta),j(\eta)}}j(\varphi)$. Since the length of all conjunctions and disjunctions in $\varphi$ is bounded by $\lambda<\kappa$, $j(\varphi)=\overline\varphi$ is the renamed version of $\varphi$, which means that $M$ satisfies that $j(N)\models_{\cL_{\eta,\eta}}\overline\varphi$, and by closure of $M$, it is correct about it. It follows that $j(N)\models_{\cL_{\eta,\eta}}\varphi$ as the $\tau$-structure via our renaming. Since $|N|=\gamma$, by elementarity, $M$ satisfies that $|j(N)|=j(\gamma)\geq j(\delta)\geq \overline\gamma^+$. Thus, in $V$, $|j(N)|>\overline\gamma$.
\end{proof}

\begin{theorem}
	If $\ULS(\mathbb L_{\eta, \eta}) = \eta$, then $\eta$ is supreme for tallness.
\end{theorem}
	\begin{proof}
		Because $\mathbb L_{\eta, \eta}$ can define well-foundedness and all ordinals $< \eta$, it is easy to see that $\eta$ is either measurable or a limit of measurables. In particular it follows that $\eta$ is a strong limit cardinal. Now let $\lambda < \eta$ and let $\theta > \eta$ be an ordinal with $\theta^\omega > \theta$. We need to find a cardinal $\lambda < \kappa \leq \eta$ that is $\theta$-tall pushing up $\eta$ with closure $\lambda$. Take the smallest ordinal $\nu > \eta$ of cofinality $> \lambda$. Notice that $\nu < \eta^+$. We produce an embedding $j: V_{\nu}\rightarrow N$ with $\lambda<\crit(j)\leq\eta$, $j(\eta) > \theta$ and such that $N^\lambda \subseteq N$. By Corollary~\ref{CharPushingUp2} this is sufficient. For $\xi \leq \lambda$ take constant symbols $c_{\xi}$, let $c_{\xi}^M = \xi$ and consider the structure $M = (V_{\nu}, \in, \eta, c_\xi^M, \Tr)_{\xi \leq \lambda}$, where $\Tr$ is a truth predicate for $(V_{\nu}, \in)$. Notice that because $\nu$ has cofinality greater than $\lambda$, $V_\nu$ is closed under $\lambda$-sequences. Then $M$ satisfies the sentence $\varphi$ of $\cL_{\eta, \eta}$ which is the conjunction of the following sentences:
		\begin{enumerate}
			\item I am well-founded.
			\item Tr is a truth predicate.
			\item $\eta$ is the largest cardinal.
			\item $\psi_\lambda$.
			\item $\bigwedge_{\xi \leq \lambda} \varphi_\xi^\psi(c_\xi)$ (where $\psi:=y\in x$).
		\end{enumerate}
		Since $\eta = \text{ULS} (\mathbb L_{\eta, \eta})$, there is a model $N = (N, E, \bar \eta, c_\xi^N, \overline{\Tr})_{\xi \leq \lambda}$, with $|N|\gg \theta$ and having $M$ as a substructure. It follows that $E$ is well-founded, so we can, by collapsing, assume that $\in = E$, $N$ is transitive and $j: M \rightarrow N$ is an elementary embedding. Since $|N|\gg \theta$, we have that $\eta<\theta<j(\eta)=\bar\eta$. In particular, \hbox{$\text{crit}(j) \leq \eta$}. Because $N \models_{\cL_{\eta,\eta}} \bigwedge_{\xi \leq \lambda} \varphi_\xi^\psi(c_\xi)$, it follows that $c_\xi^N=\xi$, and so $j(\xi)=\xi$ for all $\xi\leq\lambda$. It follows that $\text{crit}(j) > \lambda$. Finally, $N \models \psi_\lambda$ and therefore \hbox{$N^\lambda\subseteq N$}.
	\end{proof}
\begin{corollary}
	For regular cardinals $\eta$, $\ULS(\cL_{\eta, \eta}) = \eta$ if and only if $\eta$ is supreme for tallness.
\end{corollary}
Next, we consistently separate the existence of $\ULS(\cL_{\eta,\eta})$ and $\SULS(\cL_{\eta,\eta})$.
\begin{theorem}\label{th:ULSnotSULS}
It is consistent that $\eta$ is an inaccessible cardinal, $\ULS(\cL_{\eta,\eta})$ exists, but $\SULS(\cL_{\eta,\eta})$ does not exist.
\end{theorem}
\begin{proof}
Let $\eta$ be a supercompact cardinal with an inaccessible cardinal $\nu$ above it, and assume that $\nu$ is the least such inaccessible. Then $\eta$ is a limit of strong cardinals, and hence a limit of tall cardinals. In $V_\nu$, $\eta$ is also a supercompact limit of tall cardinals. Thus, we can assume without loss of generality that \hbox{$V=V_\nu$}, so that there are no inaccessible cardinals above $\eta$. First, we go to a forcing extension $V[c]$ by Cohen forcing. Since small forcing preserves tall cardinals and supercompact cardinals by standard embedding lifting arguments, $\eta$ is still a supercompact limit of tall cardinals in $V[c]$. Next, we go to a forcing extension $V[c][G]$ by $\Add(\eta,1)$. The forcing $\Add(\eta,1)$ is ${<}\eta$-closed and hence, in particular, ${\leq}\kappa$-distributive for every tall cardinal $\kappa<\eta$ in $V[c][G]$. Thus, every tall cardinal $\kappa<\eta$ remains tall in $V[c][G]$ (Theorem~3.1 in \cite{Hamkins:TallCardinals}). The cardinal $\eta$ remains inaccessible by the closure of $\Add(\eta,1)$, but since the Cohen forcing makes $\eta$ super destructible, it is not even weakly compact in $V[c][G]$ \cite{Hamkins:smallForcingSuperdestructible}. In particular, $\eta$ is not tall. Thus, in $V[c][G]$, $\eta$ cannot be $\SULS(\cL_{\eta,\eta})$, and since there are no inaccessible cardinals above $\eta$, $\SULS(\cL_{\eta,\eta})$ doesn't exist. But since $\eta$ is a limit of tall cardinals in $V[c][G]$, it is, in particular, supreme for tallness there, and hence, in $V[c][G]$, $\eta=\ULS(\cL_{\eta,\eta})$.
\end{proof}
Next, we show that consistently we can have $\ULS(\cL_{\eta,\eta})<\SULS(\cL_{\eta,\eta})$.
\begin{theorem}
It is consistent that $\eta$ is an inaccessible cardinal, $\ULS(\cL_{\eta,\eta})$ and $\SULS(\cL_{\eta,\eta})$ both exists, and $\ULS(\cL_{\eta,\eta})<\SULS(\cL_{\eta,\eta}).$
\end{theorem}
\begin{proof}
We will argue as in the proof of Theorem~\ref{th:ULSnotSULS}, but start with a model in which there is a supercompact $\eta$ and a tall cardinal $\nu$ above the supercompact cardinal. We again go to the forcing extension $V[c][G]$, in which $\eta=\ULS(\cL_{\eta,\eta})$, but $\SULS(\cL_{\eta,\eta})\neq\eta$. Next, observe that since tall cardinals are preserved by small forcing and $\Add(\eta,1)$ is small relative to $\nu$, it remains a tall cardinal in $V[c][G]$. Thus, $\SULS(\cL_{\eta,\eta})$ exists.
\end{proof}
\section{Cardinal correct extendible cardinals}\label{sec:cce}
In this section, we introduce new large cardinal notions: cardinal correct extendible cardinals and their variants. We will see in Sections \ref{sec:StrongCompactnessLI} and \ref{sec:ULSLI} that these large cardinals arise naturally from trying to characterize strong compactness cardinals and the $\ULS$ numbers for the logic $\cL(I)$. We will say that a transitive model $M$ of set theory is \emph{cardinal correct} if whenever $M$ believes that $\gamma$ is a cardinal, then $\gamma$ is a cardinal of $V$.
\begin{definition}
A cardinal $\kappa$ is \emph{cardinal correct extendible} if for every $\alpha>\kappa$, there is an elementary embedding $j:V_\alpha\to M$ with $\crit(j)=\kappa$, $j(\kappa)>\alpha$ and $M$ cardinal correct. A cardinal $\kappa$ is \emph{weakly cardinal correct extendible} if we remove the requirement that $j(\alpha)>\kappa$.
\end{definition}
Clearly, extendible cardinals are cardinal correct extendible because the rank initial segments $V_\beta$ are always cardinal correct. Recall that the property that $j(\kappa)>\alpha$ comes for free in the case of extendible cardinals, but it is not clear whether this is the case for cardinal correct extendibles.
\begin{question}
Are weakly cardinal correct extendible cardinals and cardinal correct extendible cardinals equivalent?
\end{question}
Next, we show how cardinal correct extendible cardinals and their weak version are related to strongly compact cardinals.

Recall that a cardinal $\kappa$ is \emph{$\lambda$-compact} for some $\lambda\geq\kappa$ if there is a fine $\kappa$-complete ultrafilter on $P_\kappa(\lambda)$ and a cardinal $\kappa$ is \emph{strongly compact} if it is $\lambda$-compact for every $\lambda\geq\kappa$. By a theorem of Ketonen, if $\kappa$ and $\lambda$ are regular, then $\kappa$ is $\lambda$-compact if and only if every regular $\alpha$ in the interval $[\kappa,\lambda]$ carries a uniform $\kappa$-complete ultrafilter \cite[Theorem 5.9]{Ketonen:strongCompactness}. As Goldberg points out in \cite{Goldberg:cardinalPreservingEmbeddings}, using a theorem of Kunen and Prikry, if $\lambda$ is a successor cardinal, it suffices to show this only for successor cardinals in the interval $[\kappa,\lambda]$. Kunen and Prikry showed that if $\kappa$ is regular and $U$ is a $\kappa^+$-descendingly incomplete ultrafilter on some set, then $U$ is already $\kappa$-descendingly incomplete \cite[Theorem 0.2]{KunenPrikry:uniformUltrafilters}. An ultrafilter $U$ is $\delta$-\emph{descendingly incomplete} if there is a decreasing sequence of sets in $U$ whose intersection is empty. It is easy to see that if an ultrafilter $U$ is  $\kappa$-complete and $\delta$-descendingly incomplete, then there is a uniform $\kappa$-complete ultrafilter $W\leq_{RK}U$ on $\delta$. Now suppose there is a uniform $\kappa$-complete ultrafilter on $\beta^+$ that is $\kappa$-complete. In particular, $U$ is $\beta^+$-descendingly incomplete, and hence by Kunen and Prikry's theorem, it is $\beta$-descendingly incomplete. Thus, by our earlier observation, there is a uniform $\kappa$-complete ultrafilter $W\leq_{RK}U$ on $\beta$.
Thus, we get:
\begin{theorem}\label{th:lambdaCompactness}
If $\kappa$ is regular and $\lambda$ is a successor, then $\kappa$ is $\lambda$-compact if and only if every successor cardinal in the interval $[\kappa,\lambda]$ carries a uniform $\kappa$-complete ultrafilter.
\end{theorem}

We will use the following lemma which is implicit in \cite{Goldberg:cardinalPreservingEmbeddings}.

\begin{lemma}[Goldberg]\label{le:strongCompactnessEmbedding}
	If $j: V_\lambda \rightarrow M$ is an elementary embedding with $M$ cardinal correct, $\crit(j) = \kappa$ and such that $\lambda = \sup\{j^n(\kappa) \colon n \in \omega\}$, then $j(\kappa)$ is inaccessible and $V_{j(\kappa)} \models ``\kappa \text{ is strongly compact}"$.
\end{lemma}

\begin{proposition}\label{prop:ccEimpliesStrongCompact}
If $\kappa$ is cardinal correct extendible, then $\kappa$ is strongly compact.
\end{proposition}
\begin{proof}
By Theorem~\ref{th:lambdaCompactness}, it suffices to argue that every successor cardinal $\beta^+>\kappa$ carries a uniform $\kappa$-complete ultrafilter. Let $\alpha>\beta^+$ and take $j:V_\alpha\to M$ with $\crit(j)=\kappa$, $j(\kappa)>\alpha$ and $M$ cardinal correct. Since $j(\kappa)>\alpha>\beta^+$, it follows that $j(\beta)>\beta$. Because $M$ is cardinal correct, we have that $j(\beta)$ is a cardinal and $$(j(\beta)^+)^M=j(\beta^+) = j(\beta)^+ > \beta^+.$$ This means that $j(\beta^+)$ is regular, and so, in particular, $j$ is discontinuous at $\beta^+$, i.e., $j \image \beta^+$ is bounded in $j(\beta^+)$. Thus, we can let $\gamma = \sup(j \image \beta^+) < j(\beta^+)$. It is then easy to check that we can define a uniform $\kappa$-complete ultrafilter $U$ on $\beta^+$ by letting $X \in U \text{ if and only if } \gamma \in j(X)$.

\end{proof}

\begin{theorem}
If $\kappa$ is weakly cardinal correct extendible, then $\kappa$ is strongly compact or there is an inaccessible cardinal $\alpha$ such that $V_\alpha$ satisfies that there is a strongly compact cardinal.
\end{theorem}
\begin{proof}
Suppose that $\kappa$ is not strongly compact. Choose some successor $\lambda>\kappa$ such that $\kappa$ is not $\lambda$-compact. Let $j:V_{\lambda^+}\to M$ be an elementary embedding with $\crit(j)=\kappa$ and $M$ cardinal correct. If $j(\kappa)\geq\lambda$, then the argument of the proof of Proposition~\ref{prop:ccEimpliesStrongCompact} would show that $\kappa$ is $\lambda$-compact. Thus, we have $j(\kappa)<\lambda$. This means we can apply $j$ to $j(\kappa)=\gamma$ to get $j^2(\kappa)=j(\gamma)$. Again, assuming that $j(\gamma)\geq\lambda$, we will argue that $\kappa$ must be $\lambda$-strongly compact, and so will be able to conclude that $j(\gamma)<\lambda$. By the same argument as before, we get a discontinuity for successors of $\gamma\leq\beta<\lambda$. But if $\kappa\leq \beta<\gamma$, then $\beta<\gamma=j(\kappa)\leq j(\beta)$. Repeating this argument, we get that $j^n(\kappa)<\lambda$ for all $n<\omega$. Letting $\gamma = \sup\{j^n(\kappa) \mid n < \omega \}$ we get that $j$ restricts to $j: V_\gamma \rightarrow V_\gamma^M$ and the latter is cardinal correct, because $M$ is. By Lemma~\ref{le:strongCompactnessEmbedding}, $j(\kappa)$ is inaccessible and $V_{j(\kappa)}$ satisfies that there is a strongly compact cardinal. Thus, we proved what we promised.
\end{proof}
Thus, either a cardinal correct extendible $\kappa$ is strongly compact or there is some ordinal $\lambda$ such that for cofinally many $\alpha>\kappa$, if $j:V_\alpha\to M$ with $\crit(j)=\kappa$ and $M$ cardinal correct, then $j^n(\kappa)<\lambda$ for all $n<\omega$. It is not clear whether this situation is consistent.

The following variant of cardinal correct extendibility will be especially significant in characterizing the strong compactness cardinals and $\ULS$ numbers for the logic $\cL(I)$. Notice the similarity with our notion of tall cardinals pushing up some ordinal $\delta$.
\begin{definition}
A cardinal $\kappa\leq\delta$ is \emph{cardinal correct extendible pushing up $\delta$} if for every $\alpha>\delta$, there is an elementary embedding $j:V_\alpha\to M$ with $\crit(j)=\kappa$, $M$ cardinal correct, and $j(\delta)>\alpha$.
\end{definition}
Clearly, $\kappa$ is cardinal correct extendible pushing up $\kappa$ if and only if it is cardinal correct extendible. Also, if $\kappa$ is cardinal correct extendible pushing up $\delta$, then it is weakly cardinal correct extendible.

The following lemma will be used to separate supercompact cardinals from cardinal correct extendible cardinals.
\begin{lemma}\label{lem:cceGCH}
If $\kappa$ is cardinal correct extendible pushing up $\delta$ and the $\GCH$ fails at some cardinal $\gamma\geq\delta$, then the $\GCH$ fails cofinally often.
\end{lemma}
\begin{proof}
Suppose $\delta\leq\gamma$ and $2^\gamma>\gamma^+$. Fix an ordinal $\lambda>2^\gamma$ and let $\alpha>\lambda$. We will show that the $\GCH$ fails somewhere above $\lambda$. Let $j:V_\alpha\to M$ with $\crit(j)=\kappa$, $M$ cardinal correct, and $j(\delta)>\alpha$. We have $j(\gamma)\geq j(\delta)>\alpha>\lambda$. Since $V_\alpha\models 2^\gamma>\gamma^+$, by elementarity, $M\models 2^{j(\gamma)}>j(\gamma)^+$. Since $M$ is cardinal correct, the $j(\gamma)^+$ of $M$ is the real $j(\gamma)^+$ and $(2^{j(\gamma)})^M$ is a cardinal. Thus, the $\GCH$ really must fail at $j(\gamma)^+$. Since $\lambda$ was chosen arbitrarily, it follows that the $\GCH$ fails cofinally often.
\end{proof}
\begin{theorem}\label{th:cceAboveSupercompact}
It is consistent, relative to an extendible cardinal, that there is an extendible cardinal and for every pair $\kappa\leq\delta$ such that $\kappa$ is cardinal correct extendible pushing up $\delta$, $\delta$ is bigger than the least supercompact cardinal.
\end{theorem}
\begin{proof}
We can force the $\GCH$ to hold at all regular cardinals while preserving an extendible cardinal \cite{Traprounis:ExtendiblesGCH}. So we can suppose that we start in a model $V$ in which the $\GCH$ holds at all regular cardinals and there is an extendible cardinal $\chi$. Let $\nu$ be the least supercompact cardinal, and note that $\nu<\chi$.  We force with the Laver preparation \cite{laver:supercompact} to make the supercompactness of $\nu$ indestructible by all ${<}\nu$-directed closed forcing and let $V[G]$ be the resulting forcing extension. Since the Laver preparation has size $\nu$, the $\GCH$ still holds above $\nu$ in $V[G]$ (while it fails badly below $\nu$) and $\chi$ remains extendible. Now fix any cardinal $\nu<\gamma<\chi$ and force with $\Add(\gamma,\gamma^{++})$, the forcing to add $\gamma^{++}$-many Cohen subsets to $\gamma$, and let $V[G][g]$ be the forcing extension. By the indestructibility of $\nu$ in $V[G]$, since $\Add(\gamma,\gamma^{++})$ is ${<}\nu$-directed closed, $\nu$ remains supercompact in $V[G][g]$. Also, the $\GCH$ holds above $\gamma$ in $V[G][g]$ and $\chi$ remains extendible. Thus, by Lemma~\ref{lem:cceGCH}, in $V[G][g]$, there cannot be a pair $\kappa\leq\delta$ such that $\delta\leq\gamma$ and $\kappa$ is cardinal correct extendible pushing up $\delta$.
\end{proof}
\begin{corollary}
It is consistent that a supercompact cardinal is not cardinal correct extendible.\footnote{This result was first pointed out to us by Alejandro Poveda with a more complicated argument using Radin forcing.}
\end{corollary}

\section{Strong compactness cardinals for $\cL(I)$}\label{sec:StrongCompactnessLI}

In this section, we introduce a large cardinal variant of cardinal correct extendibles characterizing being a strong compactness cardinal for the logic $\cL(I)$.

In order to obtain our theorem, we first need to argue that under certain circumstances, we can expresses well-foundedness in the logic $\cL(I)$. Unlike other logics considered so far, we cannot generally express well-foundedness in the logic $\cL(I)$. However, it turns out that we can express well-foundeness over models of a sufficiently large fragment of set theory that are cardinal correct in the sense described in Section~\ref{sec:abstractLogics}.

Suppose that $(M,\mathsf E)$ is a model of some sufficiently large fragment of ${\rm ZFC}$. Given $\xi,\eta\in \Ord^M$, we shall say that $\xi<^M\eta$ if $M\models\xi\,\mathsf E\,\eta$. We shall say that $M$ is \emph{cardinal correct} if whenever $\kappa\in \text{Card}^M$, then for every $\alpha\in \Ord^M$ with $\alpha <^M\kappa$, $$|\{\xi\in \Ord^M\mid\xi<^M \alpha\}|^V<|\{\xi\in\Ord^M\mid\xi<^M\kappa\}|^V.$$  Note that this definition of cardinal correctness reduces to the definition we gave in Section~\ref{sec:cce} for transitive models of set theory. Clearly, in the logic $\cL(I)$, there is a formula $\varphi_{\mathsf{cor}}$ expressing that $M$ is cardinal correct.

Let ${\rm ZFC}^*$ be a sufficiently large finite fragment of $\ZFC$ such that:
\begin{enumerate}
	\item $\ZFC^*$ proves:
	\begin{itemize}
		\item the ordinals form a well-ordered class,
		\item every set can be well-ordered,
		\item for every ordinal there is an increasing sequence of cardinals of that order-type,
		\item the sets are the union of the von Neumann hierarchy.
	\end{itemize}
	\item If $\lambda=\aleph_\lambda$ is an $\aleph$-fixed point, then $V_\lambda\models\ZFC^*$.
\end{enumerate}

The following result was pointed out to us by Gabe Goldberg.
\begin{theorem}[Goldberg]\label{th:cardinalityQuantifierWellFounded}
	If $(M,\mathsf E)\models{\rm ZFC}^*$ and $(M,\mathsf E)\models_{\cL(I)}\varphi_{\mathsf{cor}}$, then $\mathsf E$ is well-founded.
\end{theorem}
\begin{proof}
	Suppose that $\mathsf E$ is not well-founded. Since the sets are the union of the von Neumann hierarchy, every set has a rank and therefore, it suffices to show that the ordinals of $M$ are well-founded. Let $\kappa$ be a cardinal in $M$ such that the set of $\mathsf E$-predecessors of $\kappa$ is not well-founded. Let $\{\alpha_\xi\mid\xi<^M\kappa\}$ be a sequence with $\alpha_\xi\in\text{Card}^M$ which exists by our assumption. For $\xi<^M\kappa$, let $$\alpha_\xi^*=|\{\eta\in\Ord^M\mid \eta<^M\alpha_\xi\}|^V$$ be the cardinality in $V$. Since $M\models\varphi_{\mathsf{cor}}$, we must have that for $\xi<^M\eta<^M\kappa$, $\alpha_\xi^*<\alpha_\eta^*$ as ordinals in $V$, but this means that there is an $\omega$-descending sequence in the actual ordinals of $V$, which is impossible.
\end{proof}

\begin{theorem}\label{thm:StrongCompactnessLI}
	The following are equivalent for a cardinal $\delta$:
	\begin{itemize}
		\item[(a)] $\delta$ is a strong compactness cardinal for $\cL(I)$.
		\item[(b)] For every $\gamma > \delta$ there is $\alpha > \gamma$, a transitive set $M$ and an elementary embedding $j: V_\alpha \rightarrow M$ such that $M$ is cardinal correct, $\crit(j) \leq \delta$ and there exists $d \in M$ such that $j \image \gamma \subseteq d$ and $M \models |d| < j(\delta)$.
	\end{itemize}
	\begin{proof}
		First assume (a) and fix $\gamma > \delta$. Take any $\aleph$-fixed point $\alpha > \gamma$. Let $\tau$ be a language consisting of a binary relation $\in$ and constants $\{c_x \mid x \in V_\alpha \} \cup \{d\}$. Let $T$ be the $\cL(I)(\tau)$-theory consisting of the following:
		\begin{enumerate}
			\item ${\rm ED}_{\cL(I)}(V_\alpha,\in,c_x)_{x\in V_\alpha}$ (each $c_x$ is interpreted by the associated set $x$)
			\item $\{c_\xi\in d\mid \xi<\gamma\}$
			\item $|d|<c_\delta$
		\end{enumerate}
		Clearly, the theory $T$ is ${<}\delta$-satisfiable, as witnessed by the model $(V_\alpha,\in,c_x)_{x\in V_\alpha}$. Thus, since $\delta$ is a strong compactness cardinal of $\cL(I)$, we can fix a model $$\mathcal N=(N,{\mathsf E},c_x,d)_{x\in V_\alpha}\models T.$$ The model $N$ is well-founded by Theorem~\ref{th:cardinalityQuantifierWellFounded} because $\ZFC^*$ and $\varphi_{\mathsf{cor}}$ are in the elementary diagram. It also follows that $N$ is cardinal correct. Thus, associating $N$ with its collapse and ${\mathsf E}$ with $\in$, we get an  elementary embedding $j:V_\alpha\to N$ with $\crit(j)\leq\delta$. Since $\mathcal N\models T$, it follows that $|d|<j(\delta)$ and $j\image\gamma\subseteq d$.
		
		Next, we assume (b).  Let $\tau$ be a language and let $T$ be a ${<}\delta$-satisfiable $\cL(I)(\tau)$-theory. Assume without loss of generality that $|T| = \gamma > \delta$. Let $\gamma'>\gamma$ be large enough that $T,\tau,\delta\in V_{\gamma'}$. By (b), take $\alpha > \gamma'$ and $j:V_\alpha\to M$ such that $M$ is transitive and cardinal correct, $\crit(j) \leq \delta$ and there is a $d \in M$ such that $j\image T\subseteq d$ and $|d|^M<j(\delta)$. We can assume without loss of generality that $d$ is a $\cL(I)(j(\tau))$-theory. Thus, $M$ has a $j(\tau)$-structure $N$, which it thinks satisfies the $\cL(I)(j(\tau))$-theory $d$. Since $M$ is cardinal correct, $N$ is a really is a model of $d$. Thus, in particular, (the reduct of) $N$ is a $j\image\tau$-structure which satisfies $j\image T$. As usual, via the renaming given by $j$, $N$ is a $\tau$-structure satisfying $T$.
	\end{proof}
\end{theorem}

Notice that, as there is a proper class of $\gamma > \delta$ but only boundedly many possible critical points $\leq \delta$, if $\delta$ is a strong compactness cardinal for $\cL(I)$, then there is a fixed $\kappa \leq \delta$ such that for any $\gamma > \delta$ there is $\alpha > \gamma$ and an elementary embedding $j: V_\alpha \rightarrow M$ with $M$ transitive and cardinal correct, $\crit(j) = \kappa$ and $d \in M$ such that $j \image \gamma \subseteq d$ and $|d|^M < j(\delta)$. We therefore get:

\begin{corollary}\label{cor:StrongCompactnessLI}
	If $\delta$ is a strong compactness cardinal for $\cL(I)$, then there is $\kappa\leq\delta$ such that $\kappa$ is cardinal correct extendible pushing up $\delta$ and $\delta$ is $\kappa$-strongly compact.
	\begin{proof}
		Take $\kappa \leq \delta$ as a critical point of unboundedly many embeddings witnessing (b) of Theorem~\ref{thm:StrongCompactnessLI} as pointed out above. Then clearly, $\kappa$ is cardinal correct extendible pushing up $\delta$. And it follows from the elementary embedding characterization of $\kappa$-strongly compact cardinals in \cite[Theorem 4.7]{BagariaMagidor:groupRadicals} and (b) that $\delta$ is $\kappa$-strongly compact.
	\end{proof}
\end{corollary}
\begin{corollary}
	Consistency of a strong compactness cardinal for $\cL(I)$ implies the consistency of a strongly compact cardinal.
\end{corollary}
\begin{corollary}
	It is consistent that the least strong compactness cardinal for $\cL(I)$ is above the least supercompact cardinal.\footnote{This result was first observed by Will Boney and the second author using a different argument.}
\end{corollary}
\begin{proof}
	Consider the model from Theorem~\ref{th:cceAboveSupercompact}. Since it has an extendible cardinal, we know that there is a strong compactness cardinal for $\cL(I)$. But then the least such cardinal must be above the least supercompact cardinal $\nu$.
\end{proof}
\begin{question}
If $\delta$ is a cardinal and there is a cardinal correct extendible cardinal $\kappa\leq\delta$ pushing up $\delta$, then is $\delta$ a strong compactness cardinal for the logic $\cL(I)$?
\end{question}

For a cardinal $\delta$, the logic $\cL_{\delta, \delta}(I)$ is obtained by adding conjunctions and disjunctions and first-order quantifiers of size $< \delta$ to $\cL(I)$. Considering strong compactness cardinals of this logic gives us a sharper version of Corollary \ref{cor:StrongCompactnessLI}.
\begin{theorem}
	If $\delta$ is a strong compactness cardinal for $\cL_{\delta, \delta}(I)$, then $\delta$ is cardinal correct extendible. Moreover, the cardinal correct extendibility embeddings also witness the strong compactness of $\delta$.
	\begin{proof}
Since every ordinal ${<}\delta$ is definable in the logic $\cL_{\delta, \delta}(I)$, we can use the same argument as in the proof of Theorem \ref{thm:StrongCompactnessLI} to show that for every $\gamma > \delta$ there is $\alpha > \gamma$, a transitive set $M$ and an elementary embedding $j: V_\alpha \rightarrow M$ such that $M$ is cardinal correct, $\crit(j) = \delta$ and there exists $d \in M$ such that $j \image \gamma \subseteq d$ and $M \models |d| < j(\delta)$.  
	\end{proof}
\end{theorem}
\begin{question}
If $\delta$ is cardinal correct extendible, then is $\delta$ a strong compactness cardinal for $\cL_{\delta, \delta}(I)$?
\end{question}

\section{The logic $\cL(I)$}\label{sec:ULSLI}
In this section, we show that a strongly compact cardinal is a lower bound on the consistency strength of the existence of a ULST number for the logic $\cL(I)$ and that $\ULS(\cL(I))$ may be above the least supercompact cardinal. We show that if there is a pair $\kappa\leq\delta$ such that $\kappa$ is cardinal correct extendible pushing up $\delta$, then $\SULS(\cL(I))$ exists and is bounded by $\delta$. Almost conversely, we show that if $\ULS(\cL(I))$ exists, then there is a pair $\kappa\leq\gamma$ such that $\kappa$ is cardinal correct extendible pushing up $\gamma$.

\begin{theorem}
 If there exists a pair $\kappa\leq\delta$ such that $\kappa$ is cardinal correct extendible pushing up $\delta$, then $\SULS(\cL(I))$ exists and is at most $\delta$.
 	\begin{proof}
 		Let $N$ be a $\tau$-structure of size $\geq \delta$ and $\bar \delta > \delta$. Take $\alpha > \bar \delta^+$ with $N \in V_\alpha$ and an elementary embedding $j: V_\alpha \rightarrow M$ with $\crit(j) = \kappa$ and $j(\delta) > \alpha$ such that $M$ is cardinal correct. Consider the $j(\tau)$-structure $j(N) \in M$. Since $j \image \tau \subseteq j(\tau)$, $j(N)$ is a $\tau$-structure modulo the renaming which takes $\tau$ to $j \image \tau$. Consider the $\tau$-substructure $\bar N = j \image N \subseteq j(N)$. The bijection $j$ witnesses that $N \cong \bar N$ as $\tau$-structures. Because $M$ is cardinal correct, and thus correct about satisfaction of formulas of $\cL(I)$, it is easily seen that $\bar N \prec_{\cL(I)} j(N)$. Since $|N| \geq \delta$, in $M$, by elementarity, we have $|j(N)| \geq j(\delta) > \alpha > \bar \delta^+$, as desired.
 	\end{proof}
\end{theorem}

\begin{theorem}\label{th:ULSimpliesCardCorrect}
If $\ULS(\cL(I))$ exists, then there is a pair $\kappa\leq\gamma$ such that $\kappa$ is cardinal correct extendible cardinal pushing up $\gamma$.
\end{theorem}
\begin{proof}
Let $\delta=\ULS(\cL(I))$. Suppose $\alpha>\delta$ is an ordinal of cofinality $\omega$, as witnessed by the cofinal sequence $\la \alpha_n\mid n<\omega\ra$ of ordinals below $\alpha$. Let $\rho_\alpha$ be the least $\aleph$-fixed point above $\alpha$. Consider the structure $$\mathcal M=(V_{\rho_\alpha},\in, \la \alpha_n\mid n<\omega\ra,\alpha,\Tr).$$  Then $\mathcal M$ satisfies the sentence $\varphi$ in the logic $\cL(I)$, which is the conjunction of the sentences:
\begin{enumerate}
\item $\ZFC^*$
\item $\varphi_{\mathsf{cor}}$
\item There are no $\aleph$-fixed points above $\alpha$.
\item $\Tr$ is a truth predicate.
\item $\la \alpha_n\mid n<\omega\ra$ is cofinal in $\alpha$.
\end{enumerate}
Since $\delta=\ULS(\cL(I))$, there is a model $$\mathcal N=(N,{\mathsf E},\la \overline\alpha_n\mid n<\omega\ra,\overline\alpha,\overline{\Tr})$$ of size much larger than $\rho_\alpha$ with $\mathcal M$ as a substructure.  It follows, by Theorem~\ref{th:cardinalityQuantifierWellFounded}, that $\mathsf E$ is well-founded, and hence by collapsing, we can assume without loss of generality that $\mathsf E=\in$, $N$ is transitive, and we get an elementary embedding \hbox{$j_\alpha:V_{\rho_\alpha}\to N_\alpha$} with $\crit(j)\leq\alpha$ and $j(\alpha_n)=\overline\alpha_n$ for every $n<\omega$. Since \hbox{$\text{cof}(\alpha)=\omega$}, it follows that $\kappa_\alpha=\crit(j_\alpha)<\alpha$. We also have that $\mathcal N$ satisfies that the sequence $\la \overline\alpha_n\mid n<\omega\ra$ is cofinal in $\bar\alpha$. Now since $\bar\alpha\gg\alpha$, it follows that some $\overline\alpha_n>\alpha$. Thus, we can let $\gamma_\alpha$ be the least $\alpha_n$ such that $j_\alpha(\alpha_n)>\alpha$. Next, we define the proper class function $F$ on the ordinals of cofinality $\omega$ such that $F(\alpha)$ is the least (ordinal coding a) pair $(\kappa_\alpha,\gamma_\alpha)$ arising from one of these embeddings $j_\alpha$. Since the (ordinal coding the) pair $(\kappa_\alpha,\gamma_\alpha)$ is less than $\alpha$, the function $F$ is regressive on a stationary class, and so it must be constant on a proper class by a weak version of Fodor's Lemma which holds for classes. Let $(\kappa,\gamma)$ be this constant value. Then clearly the pair $(\kappa,\gamma)$ is as desired.
\end{proof}

\begin{corollary}
	The following are equivalent:
	\begin{enumerate}
		\item There is a pair $\kappa \leq \delta$ such that $\kappa$ is cardinal correct extendible pushing up $\delta$.
		\item $\ULS(\mathbb L(I))$ exists.
		\item $\SULS(\mathbb L(I))$ exists.
	\end{enumerate}
\end{corollary}

\begin{theorem}
It is consistent, relative to an extendible cardinal, that $\ULS(\cL(I))$ is above the least supercompact cardinal.
\end{theorem}
\begin{proof}
We work in the model $V[G][g]$ from the proof of Theorem~\ref{th:cceAboveSupercompact}, where $\nu$ was the least supercompact cardinal, $\chi>\nu$ was extendible, $2^\gamma=\gamma^{++}$ for a cardinal \hbox{$\nu<\gamma<\chi$}, and the $\GCH$ held above $\gamma$. Note that since we have an extendible cardinal in this model, $\ULS(\cL(I))$ exists. Suppose that $\ULS(\cL(I))\leq\nu$. Consider the model $(V_\rho,\in,\gamma)$, where $\rho$ is the least $\aleph$-fixed point above $\gamma$. Then by our usual arguments, there is a model $\mathcal N=(N,\in,\overline\gamma)$ of size much larger than $\rho$ that is cardinal correct and we have an elementary embedding $j:V_\rho\to N$ such that $j(\gamma)=\overline\gamma\gg\gamma$. It follows, by elementarity, that $N$ satisfies that $2^{\overline\gamma}>\overline{\gamma}^+$ and it must be correct about this by cardinal correctness. Thus, we have reached a contradiction showing that $\ULS(\cL(I))>\gamma>\nu$.
\end{proof}
\begin{question}\label{ques:CardCorrectBelowULS}
If $\ULS(\cL(I))=\delta$ exists, is there a cardinal $\kappa\leq\delta$ such that $\kappa$ is cardinal correct extendible pushing up $\delta$?
\end{question}
It follows from Theorem~\ref{th:ULSimpliesCardCorrect} that if $\ULS(\cL(I))$ exists, then either there is a strongly compact cardinal or there is an inaccessible $\alpha$ such that $V_\alpha$ satisfies that there is strongly compact cardinal. Below we slightly improve this result by showing that if $\ULS(\cL(I))=\delta$, then either there is a strongly compact cardinal ${\leq}\delta$ or there is an inaccessible $\alpha$ such that $V_\alpha$ satisfies that there is a strongly compact cardinal. This appears to be a step in the direction of getting a positive answer to Question~\ref{ques:CardCorrectBelowULS}.

Suppose that $j:V\to M$ is the ultrapower by a fine $\kappa$-complete ultrafilter $\mu$ on $P_\kappa(\lambda)$. Consider the restriction $j:V_{\lambda+2}\to V_{j(\lambda)+2}^M$, and observe that it suffices to recover $\mu$ since $P(P_\kappa(\lambda))\subseteq V_{\lambda+2}$. Next, observe that by using a flat pairing function\footnote{A \emph{flat pairing function} has the property that it does not increase the rank of the pair beyond the ranks of its elements.} and coding we can assume that $V_{\lambda+2}$ is closed under functions $f:P_\kappa(\lambda)\to V_{\lambda+2}$. Such a function $f$ consists of pairs $(A,B)$ where $A\subseteq\lambda$ and $B\subseteq V_{\lambda+1}$. Each such pair is, in turn, coded by the collection of pairs $(A,b)$ where $b\in B$ so that $b\in V_{\lambda+1}$. Thus, $f$ is coded by the collection of pairs $(A,b)$ where $f(A)=B$ and $b\in B$. Using a flat pairing function, each of the pairs $(A,b)\in V_{\lambda+1}$, and therefore $f\subseteq V_{\lambda+1}$ is in $V_{\lambda+2}$. We can thus conclude that all the information about a $\lambda$-compactness embedding is already captured by the restriction of the embedding to $V_{\lambda+2}$. Finally, since every element of $V_{j(\lambda)+2}^M$ has the form $j(f)([id]_\mu)$, where $f\in V_{\lambda+2}$, we get that $V_{j(\lambda)+2}^M$ has size at most $|V_{\lambda+2}|$.

\begin{theorem}
If $\ULS(\cL(I))=\delta$, then there is a strongly compact cardinal $\kappa\leq\delta$ or there is an inaccessible cardinal $\alpha$ such that $V_\alpha$ satisfies that there is a strongly compact cardinal.
\end{theorem}
\begin{proof}
First, let's suppose that no cardinal $\kappa\leq\delta$ is $\delta^+$-compact. Consider the structure $$\mathcal M=(V_\rho,\in, \delta,\Tr),$$ where $\rho$ is the least $\aleph$-fixed point above $\delta$ and $\Tr$ is a truth predicate for $(V_\rho, \in)$.  Then $\mathcal M$ satisfies the sentence $\varphi$ in the logic $\cL(I)$, which is the conjunction of the sentences:
\begin{enumerate}
\item $\ZFC^*$
\item $\varphi_{\mathsf{cor}}$
\item There are no $\aleph$-fixed points above $\delta$.
\item $\Tr$ is a truth predicate.
\end{enumerate}
Since $\ULS(\cL(I))=\delta$, there is a model $$\mathcal N=(N,\mathsf E,\overline \delta,\overline \Tr)\models_{\cL(I)}\varphi$$ of size much larger than $\rho$ with $\mathcal M$ as a substructure. It follows by Theorem~\ref{th:cardinalityQuantifierWellFounded} that $\mathsf E$ is well-founded, and hence by collapsing, we can assume without loss of generality that $\mathsf E=\in$, $N$ is transitive, and we get an elementary embedding $$j:V_\rho\to N$$ such that $j(\delta)=\overline\delta$. Since $|N|$ is much larger than $\rho$ and $N$ believes that there are no $\aleph$-fixed points above $\overline\delta$, it follows that $\overline\delta>\delta$, which means $j$ has a critical point and $\crit(j)=\kappa\leq\delta$. If $\crit(j)=\delta$, then, since $j(\delta)\gg \delta$, we have \hbox{$\delta^+<j(\delta^+)=j(\delta)^+$} (by cardinal correctness), which can be used to obtain a uniform $\delta$-complete ultrafilter on $\delta^+$, which, by Theorem~\ref{th:lambdaCompactness}, means that $\delta$ is is $\delta^+$-compact, contradicting our assumption. Thus, $\crit(j)=\kappa<\delta$.

First, suppose that $j(\kappa)\geq\delta$, and let's argue that $\kappa$ is $\delta^+$-compact, which would contradict our assumption.  By Theorem~\ref{th:lambdaCompactness}, it suffices to show that every successor cardinal $\beta^+$ in the interval $[\kappa,\delta^+]$ carries a uniform $\kappa$-complete ultrafilter. And again, for this it suffices to show that $\beta^+<j(\beta^+)=j(\beta)^+$ because then $j$ is discontinuous at $\beta^+$, namely $j\image\beta^+$ is bounded in $j(\beta^+)$. This is clearly true for $\beta=\delta$. So fix $\kappa\leq\beta<\delta$ and consider $\beta^+$. Since $j(\beta)\geq j(\kappa)\geq\delta$, it follows that $j(\beta)>\beta$, and hence $j(\beta)^+=j(\beta^+)>\beta^+$.  This concludes the argument that $\kappa$ is $\delta^+$-compact, which is the desired contradiction, showing that $j(\kappa)<\delta$. This means we can apply $j$ to $j(\kappa)=\gamma$ to get $j^2(\kappa)=j(\gamma)$. Again, assuming that $j(\gamma)\geq\delta$, we will argue that $\kappa$ must be $\delta^+$-compact, and so will be able to conclude that $j(\gamma)<\delta$. By the same argument as before, we get a discontinuity for successors of $\gamma\leq\beta\leq \delta$. But if $\kappa\leq \beta<\gamma$, then $\beta<\gamma=j(\kappa)\leq j(\beta)$. Repeating this argument, we get that $j^n(\kappa)<\delta$ for all $n<\omega$. Thus, by Lemma \ref{le:strongCompactnessEmbedding}, we get that $j(\kappa)$ is inaccessible and  $V_{j(\kappa)}$ satisfies that there is a strongly compact cardinal. Thus, we proved what we promised.

Otherwise, for some $\kappa\leq\delta$, $\kappa$ is $\delta^+$-compact. So we can let $\lambda>\delta^+$ be the least successor cardinal such that no $\kappa\leq\delta$ is $\lambda$-compact. Let $\rho$ be the least $\aleph$-fixed point above $|V_{\lambda+2}|$. In particular, for every successor cardinal $\theta$ with $\delta^+\leq\theta<\lambda$, $V_{\rho}$ will have an embedding $j_\theta:V_{\lambda+2}\to M_\theta$ by a fine $\kappa_\theta$-complete ultrafilter on $\mathcal P_{\kappa_\theta}(\theta)$ for some $\kappa_\theta\leq\delta$. This is the case because such an embedding has size at most $|V_{\lambda+2}|$.  Consider the structure $$\mathcal M=( V_\rho,\in,\delta,\lambda, \Tr),$$ where $\Tr$ is a truth predicate for $(V_{\rho},\in)$. Then $\mathcal M$ satisfies the sentence $\varphi$ in the logic $\cL(I)$, which is the conjunction of the sentences:
\begin{enumerate}
\item $\ZFC^*$
\item $\varphi_{\mathsf{cor}}$
\item There are no $\aleph$-fixed points above $|V_{\lambda+2}|$.
\item For every successor $\delta^+\leq \theta<\lambda$, there an elementary embedding \hbox{$j_\theta:V_{\lambda+2}\to M_\theta$} with $\crit(j_\theta)=\kappa_\theta\leq\delta$, and $j_\theta\image\theta\subseteq d$ with $|d|^M<j(\kappa_\theta)$.
\item $\Tr$ is a truth predicate.
\end{enumerate}
Since $\delta=\ULS(\cL(I))$, there is a model $$\mathcal N=( N,{\rm E},\overline\delta,\overline\lambda,\overline \Tr)\models_{\cL(I)}\varphi$$ of size much larger than $\rho$ with $\mathcal M$ as a substructure. It follows by Theorem~\ref{th:cardinalityQuantifierWellFounded} that $\mathsf E$ is well-founded, and hence by collapsing, we can assume without loss of generality that $\mathsf E=\in$, $N$ is transitive, and we get an elementary embedding $$j:V_\rho\to N$$ such that $j(\delta)=\overline\delta$ and $j(\lambda)=\overline\lambda$. Since $|N|$ is much larger than $\rho$ and $N$ believes that there are no $\aleph$-fixed points above $|V_{\overline\lambda+2}|$, it follows that $\overline\lambda>\lambda$, and thus $\crit(j)\leq\lambda$.

Let's suppose that $\crit(j)=\lambda$. Then $N$ is correct about $P(\lambda)$ and $j(\kappa)=\kappa$ for every $\kappa\leq\delta$. By elementarity, $\mathcal N$ satisfies that $\overline\lambda>\lambda$ is the least successor cardinal that is a counterexample for compactness for cardinals $\kappa\leq\delta$. Thus, $\mathcal N$ satisfies that there is $\kappa\leq\delta$ that is $\lambda$-compact, so that it has a fine $\kappa$-complete ultrafilter on $P_\kappa(\lambda)$. But since $N$ has the correct powerset of $\lambda$, this object really is a fine $\kappa$-complete ultrafilter on $P_\kappa(\lambda)$ which contradicts that $\kappa$ is not $\lambda$-compact. Thus, $\chi=\crit(j)<\lambda$. The rest of the argument splits into two cases based on whether $\chi>\delta$.

We first suppose that $\delta<\chi<\lambda$. In this case, there is $\kappa\leq\delta$ such that  $N$ thinks that $\kappa$ is $\lambda$-compact and, by (4), $N$ has an elementary embedding $$j_{\overline\lambda}:V_{\overline\lambda+2}^N\to M_{\overline\lambda}$$ with $\crit(j_{\overline\lambda})=\kappa$ and $j_{\overline\lambda}\image{\lambda}\subseteq d$ in $M_{\overline\lambda}$ with $|d|^{M_{\overline\lambda}}<j_\lambda(\kappa)$. Consider the composition $$j_{\overline\lambda}\circ j:V_{\lambda+2}\to M_{\overline\lambda}$$ (since $j(V_{\lambda+2})=V_{\overline\lambda+2}^N$). Observe that $\crit(j_{\overline\lambda}\circ j)=\kappa$ by our assumption that $\delta<\chi$.

First, let's suppose that $j\image\lambda\subseteq\lambda$. If $\beta<\lambda$, then $j_{\bar\lambda}\circ j(\beta)=j_{\bar\lambda}(j(\beta))$, where $j(\beta)=\beta'<\lambda$ by our assumption. So $j_{\bar\lambda}\circ j(\beta)=j_{\bar\lambda}(\beta')$ for some $\beta'<\lambda$. Thus $j_{\bar\lambda}\circ j\image\lambda\subseteq j_{\bar\lambda}\image\lambda\subseteq d$ and $|d|^{M_{\bar\lambda}}<j_{\bar\lambda}(\kappa)=j_{\bar\lambda}\circ j(\kappa)$. Thus, we can use the embedding $j_{\bar\lambda}\circ j$ to derive a fine $\kappa$-complete ultrafilter on $P_\kappa(\lambda)$, contradicting our assumption that $\kappa$ is not $\lambda$-compact.

Thus, we can choose $\gamma<\lambda$ to be least such that $j(\gamma)\geq\lambda$. Again, we will aim to show that $\kappa$ is $\lambda$-compact, deriving a contradiction. For $\gamma\leq\beta^+\leq\lambda$, we have $\beta<j(\beta)\leq j_{\bar\lambda}\circ j(\beta)$ and so we have the desired discontinuity of $j_{\bar\lambda}\circ j$. So it remains to show that we have uniform $\kappa$-complete ultrafilters on successors $\delta<\beta^+<\gamma$. For this it suffices to show that $\kappa$ is $\beta^+$-compact, but this follows, by using $j_{\bar\lambda}\circ j$ and observing that $j_{\bar\lambda}\circ j\image \beta^+\subseteq j_{\bar\lambda}\image\lambda\subseteq d$. So we have again derived a contradiction.

Finally, assume that $\chi \leq \delta < \lambda$. Then it follows that $\chi$ is not $\lambda$-compact. Now if $j(\chi) \geq \lambda$, we get for $\chi \leq \beta^+ \leq \lambda$ that $\beta < j(\beta)$ and thus a discontinuity of $j$. In particular, $\chi$ is $\lambda$-strongly compact, which is again a contradiction. And if $j(\chi) < \lambda$, we can reason as before to show that $j^n(\chi) < \lambda$ for all $n \in \omega$. Thus, with $\gamma = \sup\{j^n(\chi) \colon n \in \omega \} \leq \lambda$, we have that $j$ restricts to $j: V_\gamma \rightarrow V_\gamma^N$ and the latter is cardinal correct by correctness of $N$ and hence, by Lemma \ref{le:strongCompactnessEmbedding}, we get that $j(\chi)$ is inaccessible and $V_{j(\chi)}$ believes that $\chi$ is a strongly compact cardinal.
\end{proof}

\bibliography{uls}

\end{document}